\documentclass{amsart}

\usepackage{tikz}
\usepackage{hyperref}
\usepackage{enumerate}

\usepackage{xcolor}
\definecolor{DeepPink}{RGB}{255,20,147} 
\definecolor{DeepPurple}{RGB}{160,5,220}

\usepackage{amssymb} 
\def\semicolon{;}
\def\applytolist#1{
    \expandafter\def\csname multi#1\endcsname##1{
        \def\multiack{##1}\ifx\multiack\semicolon
            \def\next{\relax}
        \else
            \csname #1\endcsname{##1}
            \def\next{\csname multi#1\endcsname}
        \fi
        \next}
    \csname multi#1\endcsname}

\def\calc#1{\expandafter\def\csname #1\endcsname{{\mathbb #1}}}
\applytolist{calc}QWERTYUIOPLKJHGFDAZXCVBNM;
\def\calc#1{\expandafter\def\csname bf#1\endcsname{{\mathbf #1}}}
\applytolist{calc}QWERTYUIOPLKJHGFDSAZXCVBNMqwertyuiopasfghjklzxcvbnm;
\def\calc#1{\expandafter\def\csname cal#1\endcsname{{\mathcal #1}}}
\applytolist{calc}QWERTYUIOPLKJHGFDSAZXCVBNM;
\def\calc#1{\expandafter\def\csname s#1\endcsname{{\mathscr #1}}}
\applytolist{calc}QWERTYUIOPLKJHGFDSAZXCVBNM;
\def\calc#1{\expandafter\def\csname frak#1\endcsname{{\mathfrak #1}}}
\applytolist{calc}QWERTYUIOPLKJHGFDSAZXCVBNM;
\def\calc#1{\expandafter\def\csname tb#1\endcsname{{\text{\textbf{#1}}}}}
\applytolist{calc}QWERTYUIOPLKJHGFDSAZXCVBNMqwertyuiopasdfghjklzxcvbnm;

\newcommand{\angbra}[1]{\left\langle #1\right\rangle}

\newcommand{\inj}[0]{\ar@{^{(}->}}

\newcommand{\nbar}[1]{\overline{#1}}

\newcommand{\dcup}[0]{\displaystyle\bigcup}
\newcommand{\dsqcup}{\displaystyle\bigsqcup}

\newcommand{\ph}{\varphi}

\newcommand{\toup}[1]{\stackrel{#1}{\longrightarrow}}

\newcommand{\Hom}{\operatorname{Hom}}

\newcommand{\into}{\hookrightarrow}

\newcommand{\cl}{\operatorname{cl}}

\newcommand{\bdry}{\partial}

\newcommand{\vspan}{\operatorname{span}}

\newcommand{\conv}{\operatorname{Conv}}

\newcommand{\rec}{\operatorname{rec}}

\newcommand{\scalefactor}{0.9}
\newcommand{\fncolon}{\colon}

\newcommand{\aff}{\operatorname{aff}}

\newcommand{\sdrop}{\smallsetminus}
\newcommand{\smallfan}{\Sigma}
\newcommand{\bigfan}{\Sigma'}
\newcommand{\smallcpx}{\Phi}
\newcommand{\bigcpx}{\Pi}

\newcommand{\bigCpx}{\bigcpx}
\newcommand{\subdividedCpx}{\widetilde{\bigcpx}}

\newcommand{\htzero}[2]{#1|_{#2\times\{0\}}}
\newcommand{\htone}[2]{#1|_{#2\times\{1\}}}
\newcommand{\hti}[3]{#1|_{#2\times\{#3\}}}
\newcommand{\hticurrently}[3]{\hti{#1}{#2}{#3}}
\newcommand{\cpxToComplete}{\smallcpx}

\newcommand{\completeCpx}{\nbar{\cpxToComplete}}
\newcommand{\completecpx}{\completeCpx}
\newcommand{\Face}{\operatorname{Face}}
\newcommand{\Faces}{\Face}
\newcommand{\properFaces}[1]{\Delta_{#1}}
\newcommand{\halfspaceSmallFan}{\Delta}
\newcommand{\halfspaceBigFan}{\Delta'}

\newcommand{\halfspacebigfan}{\halfspaceBigFan}

\newcommand{\subfieldOfR}{\mathbb{F}}
\newcommand{\Pzon}{\calP_{zon}}

\newcommand{\pairingWithDots}{\angbra{\cdot\,,\cdot}}

\newtheorem{theorem}{Theorem}[section]
\newtheorem{lemma}[theorem]{Lemma}

\newtheorem{proposition}[theorem]{Proposition}
\newtheorem{thm}[theorem]{Theorem}

\newtheorem{prop}[theorem]{Proposition}
\theoremstyle{definition}
\newtheorem{remark}[theorem]{Remark}
\newtheorem{example}[theorem]{Example}

\newtheorem{defi}[theorem]{Definition}
\newtheorem{construction}[theorem]{Construction}
\title{Locally finite completions of polyhedral complexes}
\author{Desmond Coles, Netanel Friedenberg}

\begin{document}

\begin{abstract}
We develop a method for subdividing polyhedral complexes in a way that restricts the possible recession cones and allows one to work with a fixed class of polyhedron. We use these results to construct locally finite completions of rational polyhedral complexes whose recession cones lie in a fixed fan, locally finite polytopal completions of polytopal complexes, and locally finite zonotopal completions of zonotopal complexes.
\end{abstract}

\maketitle

\section{Introduction}

The goal of this article is to develop a method for completing a polyhedral complex such that the recession cones of its polyhedra lie in a fixed collection, and the resulting completion is locally finite. Let $N$ be a lattice of rank $n$ and set $N_{\R}=N\otimes_{\Z}\R$. Let $\cpxToComplete$ be a polyhedral complex in $N_{\R}$ (which is not necessarily finite). We say that $\cpxToComplete$ is \textit{complete} if $\cup_{P\in \cpxToComplete} P = N_{\R}$; a \textit{completion} of $\cpxToComplete$ is a polyhedral complex $\completeCpx \supseteq \cpxToComplete$ which is complete.  Let $\rec P$ denote the recession cone of a polyhedron $P$. Our first result is the following.

\begin{thm}\label{thm:LocallyFiniteCompletions}
Let $\Gamma $ denote an additive subgroup of $\R$. Let $\smallfan$ be a finite rational fan in $N_{\R}$ and let $\cpxToComplete$ be a finite $\Gamma$-rational polyhedral complex in $N_{\R}$ such that the recession cone $\rec P$ of any $P\in\cpxToComplete$ is in $\smallfan$. 
Then there is a complete $\Gamma$-rational polyhedral complex $\completeCpx$ in $N_{\R}$ containing $\cpxToComplete$ as a subcomplex such that 
\begin{itemize}
\item $\{\rec P\mid P\in\completeCpx\}=\smallfan$ and
\item $\completeCpx$ is locally finite in the tropical toric variety $N_{\R}(\smallfan)$.
\end{itemize}
Moreover, if $\cpxToComplete$ admits a finite completion whose vertices are all in $N_{\Gamma}$ then $\completecpx$ can be chosen so that all of its vertices are in $N_{\Gamma}$.
\end{thm}

\noindent  Recall that a polyhedron $P$ in $N_{\R}$ is called \emph{$\Gamma$-rational} if it can be written in the form 
$$P=\{w\in N_{\R}\mid \angbra{u_i,w}\geq \gamma_i\text{ for }i=1,\ldots,m\}$$ 
with $u_1,\ldots,u_m\in M$ and $\gamma_1,\ldots,\gamma_m\in\Gamma$.  The tropical toric variety $N_{\R}(\smallfan)$ is a partial compactification of $N_{\R}$, see \S\ref{subsec:TropicalToricVars} for the definition of $N_{\R}(\smallfan)$. See \cite[\S3]{LimitOfTrops} for further discussion on $N_{\R}(\Sigma)$ and the connection to algebraic geometry. We will have to work with tropical toric varieties for future algebro-geometric applications discussed below. Here $N_{\Gamma}$ denotes $N\otimes_\Z \Gamma$. All $\Gamma$-rational polyhedron have vertices in $N_{\Q \Gamma}:= N \otimes_{\Z} \Q \Gamma$, where $\Q \Gamma$ is the divisible hull of $\Gamma$. Note that it is not always the case that a $\Gamma$-rational polyhedron has vertices in $N_{\Gamma}$; see Remark \ref{Rem:VerticesClarification}.

In the case where $\Gamma=\Q$, Theorem \ref{thm:LocallyFiniteCompletions} gives a result about locally finite completions of rational polyhedral complexes. In Theorem \ref{thm:LocallyFiniteKDefinableCompletions} we generalize this to complexes of polyhedra which can be defined using coefficients in a given subfield of $\R$.

The motivation for Theorem \ref{thm:LocallyFiniteCompletions} comes from algebraic geometry; in \cite{GublerSoto} a correspondence between toric varieties over a valuation ring and certain fans is established. In a companion article we will apply Theorem \ref{thm:LocallyFiniteCompletions} to construct of algebraizable formal models of algebraic varieties. Our method also allows us to construct completions while working with a fixed class of polyhedron (such as zonotopes). We prove two other results on completions, one about polytopal completions, Theorem \ref{thm:LocallyFiniteRealPolytopalCompletions}, and one about zonotopal completions, Theorem \ref{thm:ZonotopalCompletions}. Let $W$ denote an arbitrary finite dimensional real vector space; we will use $N_{\R}$ when we need to keep track of particular additive subgroups of our vector space and otherwise we will use $W$ for simpler notation.

\begin{thm}\label{thm:LocallyFiniteRealPolytopalCompletions}
Every finite polytopal complex in $W$ admits a polytopal completion that is locally finite in $W$.
\end{thm}

\begin{remark}
Theorem \ref{thm:LocallyFiniteRealPolytopalCompletions} is a polytopal generalization of Whitehead’s Completion Lemma, which says that any finite simplicial complex in a real vector space admits a locally finite completion \cite[Page 238]{ZieglerPolytopes}, \cite[Section 4]{WhiteheadSubdivisions}. We note also that Whitehead’s Completion Lemma follows from Theorem \ref{thm:LocallyFiniteRealPolytopalCompletions} by ordering the vertices and applying a pulling triangulation; this proof of the simplicial completion lemma is new and different from the traditional one.
\end{remark}

\begin{thm}\label{thm:ZonotopalCompletions}
Let $\cpxToComplete$ be a finite zonotopal complex in $W$. If $|\cpxToComplete|$ is a polytope or a star-shaped ball then $\cpxToComplete$ admits a zonotopal completion which is locally finite in $W$.
\end{thm}
See \S \ref{section:ZonotopalCompletions} for the definition of a star-shaped ball. Recall that a zonotope is a polytope which can be written as a Minkowski sum of finitely many line segments. A zonotopal complex is a polyhedral complex where every polyhedron is a zonotope, and a zonotopoal completion is a completion which is a zonotopal complex. Zonotopal decompositions of various regions have been studied both purely within the context of polyhedral geometry and combinatorics and for relations to algebraic geometry; see \cite[Lecture 7]{ZieglerPolytopes}, \cite[\S 4.2]{TilingsAndBottSamelsons}, \cite{YangBaker}, and \cite{OS}.

The paper is structured as follow. In \S \ref{sec:Preliminaries} we give our notations, conventions, and an overview of background material. Then we prove Theorems \ref{thm:LocallyFiniteCompletions}, \ref{thm:LocallyFiniteRealPolytopalCompletions}, and \ref{thm:ZonotopalCompletions} by proceeding in two parts. Let $\Phi$ be as in Theorem \ref{thm:LocallyFiniteCompletions}. First, in \S \ref{sec:LocallyFiniteSubdivisions} we establish that if we are given an extension of complexes (satisfying some conditions) $\Pi\supseteq \Phi$ then we can subdivide $\Pi$ in a way that preserves $\Phi$, restricts the possible recession cones of the polyhedron in the subdivision, and preserves local finiteness in $W(\Sigma)$. Our work on subdivisions in this section culminate in Proposition \ref{prop:MainSubdivisions}. The second part is in \S \ref{sec:LocallyFiniteCompletions} where we establish when an extension $\Pi\supseteq \Phi$, meeting the necessary requirements, exists. In this section we work in greater technical detail to ensure that our completions and subdivisions can be constructed only using polyhedra from a specific class, such as $\Gamma$-rational polyhedron or zonotopes, for example. This culminates in Theorem \ref{thm:GeneralSubdivisions} and Theorem \ref{thm:GeneralLocallyFiniteCompletions}.

\subsection*{Acknowledgments} 
The authors are grateful to Sam Payne for his encouragement on this project and for many helpful discussions and comments. They also thank Kalina Mincheva and Jeremy Usatine for helpful conversations. Netanel Friedenberg is thankful to the Mathematical Sciences Research Institute and the organizers of the Birational Geometry and Moduli Spaces program in the Spring 2019 semester during which some of the work on this project was done (Desmond Coles was not present). This paper was also supported later by NSF DMS-2001502, and NSF DMS-2053261.


\section{Preliminaries}\label{sec:Preliminaries}
We fix our notations and establish some basic facts about polyhedra, polyhedral complexes, and tropical toric varieties. We follow the same conventions as \cite{ZieglerPolytopes} and refer the reader to \cite[\S 3]{LimitOfTrops} and \cite[\S 3]{RabinoffTropicalAnalyticGeometry} for details about tropical toric varieties. 

Throughout this paper we fix a finite dimensional real vector space $W$. A polyhedron in $W$ is finite intersections of half-spaces in $W$, and a polytope is a bounded polyhedron.

\subsection{Polyhedra}
Let $W^*$ be the dual space of $W$ and let $\pairingWithDots\fncolon W^*\times W\to\R$ denote the natural pairing. Given a polyhedron $P$ in $W$ we let $\rec P$ denote its recession cone, defined as:
$$\rec P=\{v\in W \mid u+tv\in P \text{ for any $t\geq 0$ and $u\in P$} \}.$$
Write $Q\leq P$ if $Q$ is a face of $P$. We will use the following two elementary lemmas about images of polyhedra under linear maps.

\begin{lemma}\label{lemma:MapAndRecessionCommute}
Let $\ph\fncolon V\to W$ be a linear map of real vector spaces and let $P\subset V$ be a polyhedron. Then $\ph(\rec P)=\rec(\ph(P))$.
\end{lemma}\begin{proof}
Write $P=Q+\rec P$ with $Q$ a polytope in $V$. 
Then $\ph(P)=\ph(Q)+\ph(\rec(P))$ with $\ph(Q)$ a polytope and $\ph(\rec P)$ a cone, so $\rec(\ph(P))=\ph(\rec P)$.
\end{proof}

\begin{lemma}\label{lemma:FaceModRecessionGivesFace}
Let $P$ be a polyhedron in $W$, let $\sigma:=\rec P$ be the recession cone of $P$, and let $\tau$ be a face of $\sigma$. Let $\pi\fncolon W\to W/\vspan(\tau)$ be the canonical projection map. For any face $F$ of $P$ such that $\tau\leq\rec F$, $\pi(F)$ is a face of $\pi(P)$. If $F$ is a proper face of $P$ then $\pi(F)$ is a proper face of $\pi(P)$.
\end{lemma}\begin{proof}
Let $u\in W^*$ be such that $F$ is the subset of $P$ on which $u$ is minimized, and let $a$ be this minimal value. Because $u$ takes a constant value on $F$ and $\tau\leq\rec F$, we have $\tau\subset u^\perp$ and so $u\in(\vspan\tau)^{\perp}$. 
Thus $u$ descends to a linear function $\tilde{u}$ on $W/\vspan\tau$, so $u=\tilde{u}\circ\pi$. 
Hence $\angbra{\tilde{u},v}\geq a$ for $v\in\pi(P)$ and $\tilde{u}$ is constant with value $a$ on $\pi(F)$. So $\pi(F)$ is contained in the face $F':=\pi(P)\cap\{v\in W/\vspan\tau\mid \angbra{\tilde{u},v}=a\}$ of $\pi(P)$. 
We also have that 
\begin{align*}
P\cap\pi^{-1}(F')&=P\cap\pi^{-1}(\pi(P))\cap\pi^{-1}(\{v\in W/\vspan\tau\mid \angbra{\tilde{u},v}=a\})\\
&=P\cap\{w\in W\mid \angbra{u,w}=a\}=F,
\end{align*}
 and because $F'\subset\pi(P)$, $F'=\pi(P\cap\pi^{-1}(F'))$, so $F'=\pi(F)$.

If $F$ is a proper face of $P$ then picking some $w\in P\sdrop F$ we have $\angbra{u,w}>a$. Then $\pi(w)\in\pi(P)$ is such that $\angbra{\tilde{u},\pi(w)}=\angbra{u,w}>a$, so $\pi(w)\in\pi(P)\sdrop\pi(F)$ and we have that $\pi(F)$ is a proper face of $\pi(P)$.
\end{proof}

\subsection{Rationality and Definablity}

For more many of our constructions we will need consider not only a vector space but a lattice inside of it (as well as other additive subgroups). In this case we use the following notation. Let $N$ be a lattice and let $M:=\Hom(N,\Z)$ be its dual lattice. Let $\Gamma$ be an additive subgroup of $\R$ and let $\Q\Gamma$ be the divisible hull of $\Gamma$. Let $N_{\Gamma}:=N\otimes_{\Z}\Gamma$, in particular $N_{\R}=N\otimes_{\Z}\R$. We will always consider $N_{\Gamma}$ as a subgroup of $N_{\R}$. A polyhedron $P$ in $N_{\R}$ is called \emph{$\Gamma$-rational} if it can be written in the form 
$$P=\{w\in N_{\R}\mid \angbra{u_i,w}\geq \gamma_i\text{ for }i=1,\ldots,m\}$$ 
with $u_1,\ldots,u_m\in M$ and $\gamma_1,\ldots,\gamma_m\in\Gamma$. 
Polyhedra which are $\{0\}$-rational are cones, and we call such cones \emph{rational}. Every face of a $\Gamma$-rational polyhedron is $\Gamma$-rational. A polyhedron $P$ in $N_{\R}$ is $\Gamma$-rational if and only if the affine hull of each face is the translation of a rational linear subspace of $N_{\R}$ by an element of $N_{\Q\Gamma}$ (\cite[A.2]{GublerGuideTrop}). It follows that the Minkowski sum of two $\Gamma$-rational polyhedra is $\Gamma$-rational.

\begin{remark}\label{Rem:VerticesClarification}
Note that any $\Gamma$-rational polyhedron has vertices in $N_{\Q\Gamma}$, however it may not have vertices in $N_{\Gamma}$. For example, let $\Gamma$ be the group generated by $\sqrt{3}$, and considered the polyderon given by the line segment between the origin in $\R$ and $\sqrt{3}/2$.
\end{remark}

Let $\subfieldOfR$ be a subfield of $\R$. A polyhedron in $N_{\R}$ is called \emph{$\subfieldOfR$-definable} if it can be written in the form 
$\{w\in N_{\R}\mid \angbra{u_i,w}\geq a_i\text{ for }i=1,\ldots,m\}$ 
with $u_1,\ldots,u_m\in M_{\subfieldOfR}$ and $a_1,\ldots,a_m\in \subfieldOfR$. Alternatively, we can characterize the $\subfieldOfR$-definable polyhedra as follows. The $\subfieldOfR$-definable polyhedra in $N_{\R}$ are those which can be written as a Minkowski sum of an $\subfieldOfR$-definable polytope and an $\subfieldOfR$-definable cone. A polytope is $\subfieldOfR$-definable if and only if it can be written as the convex hull of finitely many points in $N_{\subfieldOfR}$, and a cone $\sigma$ is $\subfieldOfR$-definable if and only if there are $w_1,\ldots,w_m\in N_{\subfieldOfR}$ such that $\sigma=\R_{\geq0}w_1+\cdots+\R_{\geq0}w_m$.

\subsection{Polyhedral complexes}

A \emph{polyhedral complex} in $W$ is a collection $\Pi$ of polyhedra in $W$ such that 
\begin{itemize}
\item if $P\in\Pi$ and $F$ is a face of $P$ then $F\in\Pi$, and
\item if $P,Q\in\Pi$ are not disjoint then $P\cap Q$ is a face of both $P$ and $Q$.
\end{itemize}
A \emph{subcomplex} of $\Pi$ is a polyhedral complex $\Pi'\subset\Pi$ or, equivalently, a subset $\Pi'\subset\Pi$ such that if $P\in\Pi'$ and $F$ is a face of $P$ then $F\in\Pi'$. 
A \emph{fan} in $W$ is a nonempty polyhedral complex consisting of pointed cones. For a cone $\sigma$ let $\Face(\sigma)$ be the fan given by the faces of $\sigma$.
A \emph{polytopal complex} is a polyhedral complex consisting of polytopes. 
Note that we do not require any of these to be finite.

The \emph{support} of a polyhedral complex $\Pi$ is the union $|\Pi|:=\cup_{P\in\Pi}P$ of all the polyhedra in $\Pi$. 
We say that $\Pi$ is \emph{complete} if $|\Pi|=W$. 
A \emph{completion} of $\Pi$ is a complete polyhedral complex containing $\Pi$ as a subcomplex.

Let $P$ be a polyhedron. A \emph{subdivision} of $P$ is a polyhedral complex $\Pi$ such that $|\Pi|=P$. More generally, if $\Pi$ and $\Pi'$ are polyhedral complexes we say that \emph{$\Pi'$ subdivides $\Pi$} if $|\Pi'|=|\Pi|$ and every polyhedron in $\Pi$ can be written as a union of polyhedra in $\Pi'$. 
A polyhedral complex $\Pi$ subdivides a polyhedron $P$ if and only if $\Pi$ is a subdivision of the polyhedral complex $\Faces(P):=\{F\mid F\leq P\}$ of all faces of $P$.

A polyhedral complex $\Pi$ in $N_{\R}$ is \emph{$\Gamma$-rational} if every polyhedron $P\in\Pi$ is $\Gamma$-rational. 
A fan in $N_{\R}$ is \emph{rational} if each of its cones is rational.

Let $\subfieldOfR$ be a subfield of $\R$. A polyhedral complex in $N_{\R}$ is called \emph{$\subfieldOfR$-definable} if all of its polyhedra are $\subfieldOfR$-definable. Note that a fan $\Delta$ in $N_{\R}\times\R_{\geq0}$ is $\subfieldOfR$-definable if and only if $\hticurrently{\Delta}{N_{\R}}{1}$ is an $\subfieldOfR$-definable polyhedral complex and $\hticurrently{\Delta}{N_{\R}}{0}$ is an $\subfieldOfR$-definable fan.

\subsection{Tropical toric varieties}\label{subsec:TropicalToricVars} Our presentation of tropical toric varieties follows \cite[\S 3]{RabinoffTropicalAnalyticGeometry} (and \cite[\S 3]{LimitOfTrops} though this source works with rational fans).

Let $\Sigma$ be a finite fan in $W$. The \emph{tropical toric variety} $W(\Sigma)$ corresponding to $\Sigma$ is a partial compactification of $W$ constructed as follows. As a set, 
$$W(\Sigma):=\dsqcup_{\sigma\in\Sigma}W/\vspan\sigma.$$ 

Consider the extended real line $\nbar{\R}=\R\cup\{\infty\}$ with the topology such that the homeomorphism $\R\to\R_{>0}$ given by $a\mapsto e^{-a}$ extends to a homeomorphism $\nbar{\R}\to\R_{\geq0}$. We equip $\nbar{\R}$ with the addition map for which $a+\infty=\infty$ for all $a\in\nbar{\R}$, so the aforementioned homeomorphism $\nbar{\R}\to\R_{\geq0}$ is also an isomorphism of topological monoids $(\nbar{\R},+)\to(\R_{\geq0},\cdot)$.

Given $\sigma\in\Sigma$, let 
$\sigma^\vee:=\{u\in W^*\mid \angbra{u,w}\geq0\text{ for all }w\in\sigma\}$ 
be the dual cone of $\sigma$. 
Let $\Hom_{\R_{\geq0}}(\sigma^\vee,\nbar{\R})$ denote the set of monoid homomorphisms $\sigma^\vee\to\nbar{\R}$ that preserve multiplication by elements of $\R_{\geq0}$. 
We can identify $W(\sigma):=\dcup_{\tau\leq\sigma}W/\vspan \tau$ with $\Hom_{\R_{\geq0}}(\sigma^\vee,\nbar{\R})$ by sending $w\in W/\vspan \tau$ to the monoid homomorphism $\ph_w\fncolon \sigma^\vee\to\nbar{\R}$ given by 
$$u\mapsto\begin{cases}
\angbra{u,w}&\text{if }u\in\tau^{\perp}\\
\infty&\text{otherwise}.
\end{cases}$$
This gives a bijection $W(\sigma)\toup{\sim}\Hom_{\R_{\geq0}}(\sigma^\vee,\nbar{\R})$, which gives $W(\sigma)$ the structure of a topological monoid. Note that for $\tau\leq\sigma$, the subspace topology on $W/\vspan\tau\subset W(\sigma)$ is the usual topology on $W/\vspan\tau$ as a finite-dimensional $\R$-vector space. For $\tau_1,\tau_2\leq\sigma$, $w_1\in W/\vspan\tau_1$, and $w_2\in W/\vspan\tau_2$, we have that $w_1+w_2$ is given by projecting each $w_i$ to $W/\vspan(\tau_1\vee\tau_2)$, where $\tau_1\vee\tau_2$ is the smallest face of $\sigma$ containing both $\tau_1$ and $\tau_2$, and then applying the usual vector space addition in $W/\vspan(\tau_1\vee\tau_2)$.

For each $\sigma\in\Sigma$ and any finite generating set $\{u_1,\ldots,u_m\}$ for $\sigma^\vee$ as a polyhedral cone, the map $W(\sigma)=\Hom_{\R_{\geq0}}(\sigma^\vee,\nbar{\R})\to\nbar{\R}^m$ defined by $\ph\mapsto(\ph(u_1),\ldots,\ph(u_m))$ gives a homeomorphism from $W(\sigma)$ to a subspace of $\nbar{\R}^m$.

If $\sigma\in\Sigma$ and $\tau$ is a face of $\sigma$, then the inclusion $W(\tau)\subset W(\sigma)$ identifies $W(\tau)$ with an open topological submonoid of $W(\sigma)$. 
The topology on $W(\Sigma)$ is given by gluing the topologies on the various $W(\sigma)$s, and the previous sentence guarantees that each $W(\sigma)\subset W(\Sigma)$ is an open subspace whose topology is the same as the original topology on $W(\sigma)$. We consider $W(\Sigma)$ as a partial compactification of $W$ by identifying $W$ with $W/\{0\}\subset W(\Sigma)$. See Figure \ref{Fig:TropicalToricVariety} for an example.
Since $W=W/\{0\}=W(\{0\})$ is a submonoid of each $W(\sigma)$ and the inclusion maps $W(\tau)\into W(\sigma)$ are monoid homomorphisms, we get that the actions of $W$ on each $W(\sigma)$, given by addition, glue to give a continuous action $W\times W(\Sigma)\to W(\Sigma)$, which we will also write as addition. 

\begin{figure}
\begin{center}\begin{tikzpicture}[scale=1]
\fill[gray!10!white] (-2,0)-- (0,0) -- (0,2) -- (-2,2) -- cycle;
\draw(-2,0)-- (0,0);
\draw (-2,0) -- (0,2);
\draw(-2,0) -- (-2,2);
\node[below] at (-1,-0.2) {$\Sigma$};

\fill[gray!10!white] (2,2) -- (3,2) -- (4,1) -- (4,0) -- (2,0) -- cycle;
\draw(2,2) -- (3,2);
\draw(3,2) -- (4,1);
\draw(4,1) -- (4,0);
\node[below] at (3,-0.2) {$W(\Sigma)$};
\end{tikzpicture}\end{center}
\caption{The shaded region depicts $W$ and the lines depict the boundary of $W$. The picture on the right is the dual complex to the figure on the left.}\label{Fig:TropicalToricVariety}
\end{figure}
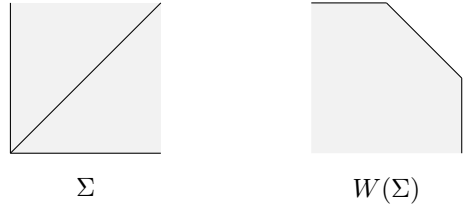

We will use the fact that tropical toric varieties are Hausdorff several times, both directly and by using \cite[Remark 3.21]{RabinoffTropicalAnalyticGeometry}, which implicitly uses that tropical toric varieties are Hausdorff. The following is well known among tropical geometers but we include the proof for completeness.

\begin{lemma}\label{lemma:TropicalToricVarsHausdorff}
Let $\Sigma$ be a finite fan in $W$. Then $W(\Sigma)$ is Hausdorff.
\end{lemma}\begin{proof}
We show that $W(\Sigma)$ is Hausdorff by tropicalizing the proof, as in \cite[Theorem 3.1.5]{CLS}, that the toric variety associated to a rational fan is separated. First, recall that, for any $\sigma\in\Sigma$, $W(\sigma)$ is Hausdorff because it is homeomorphic to $\Hom_{\R_{\geq0}}(\sigma^\vee,\nbar{\R})$ which is a subspace of the Hausdorff space $\nbar{\R}^{\left(\sigma^\vee\right)}$. 
So it suffices to show that, for any two cones $\sigma_1,\sigma_2\in\Sigma$ with intersection $\tau$, the diagonal map 
$$\Hom_{\R_{\geq0}}(\tau^\vee,\nbar{\R})=W(\tau)\to W(\sigma_1)\times W(\sigma_2)=\Hom_{\R_{\geq0}}(\sigma_1^\vee,\nbar{\R})\times\Hom_{\R_{\geq0}}(\sigma_2^\vee,\nbar{\R})$$
 has closed image. 
This diagonal map is induced by the inclusions $\sigma_1^\vee,\sigma_2^\vee\subset\tau^\vee$, and if we identify 
$\Hom_{\R_{\geq0}}(\sigma_1^\vee,\nbar{\R})\times\Hom_{\R_{\geq0}}(\sigma_2^\vee,\nbar{\R})$ 
with 
$\Hom_{\R_{\geq0}}(\sigma_1^\vee\times\sigma_2^\vee,\nbar{\R})$, 
then the diagonal map $\Hom_{\R_{\geq0}}(\tau^\vee,\nbar{\R})\to\Hom_{\R_{\geq0}}(\sigma_1^\vee\times\sigma_2^\vee,\nbar{\R})$ is induced by the addition map $\sigma_1^\vee\times\sigma_2^\vee\to\tau^\vee$. 
That this map has closed image follows because $\sigma_1^\vee+\sigma_2^\vee=\tau^\vee$ and in general, if $\psi\fncolon\calS\to\calT$ is a surjective homomorphism of semimodules over a semiring $\calR$ and $\calM$ is a Hausdorff topological $\calR$-semimodule, then the image of the induced map $\Hom_{\calR}(\calT,\calM)\to\Hom_{\calR}(\calS,\calM)$ is closed. To see this last fact, note that the image of this map is the intersection of the sets 
$$\{\ph\in\Hom_{\calR}(\calS,\calM)\mid \ph(s_1)=\ph(s_2)\}$$ 
over all pairs $s_1,s_2\in\calS$ such that $\psi(s_1)=\psi(s_2)$. Because $\calM$ is Hausdorff, each of the sets $\{\ph\in\Hom_{\calR}(\calS,\calM)\mid \ph(s_1)=\ph(s_2)\}$ is closed.
\end{proof}

Lastly, recall that a collection $\calA$ of subsets of a topological space $X$ is called \emph{locally finite in $X$} if every point of $X$ has a neighborhood which meets only finitely many of the sets $A\in\calA$. Note that $\calA$ is locally finite in $X$ if and only if the family $\{\nbar{A}\mid A\in\calA\}$ of the closures of the sets in $\calA$ is locally finite in $X$. 
If $\calA$ is locally finite in $X$ and $\calA'$ is locally finite in $X'$, then the product family $\calA\times\calA':=\{A\times A'\mid A\in\calA,A'\in\calA'\}$ is locally finite in $X\times X'$. If $Y$ is a closed subspace of $X$ which contains $\cup_{A\in\calA}A$, then $\calA$ is locally finite in $X$ if and only if it is locally finite in $Y$.


\section{Locally finite subdivisions}\label{sec:LocallyFiniteSubdivisions}

In this section we develop the subdivision technique used to prove our main theorems. We give a simple example of this technique in  Example \ref{example:One}; see Example \ref{example:Three} for an example the illustrates a more intricate case. 
In \S \ref{subsec:SubdividingPolyhedron} we show that this technique gives a locally finite subdivision of a single polyhedron in $W$. In \S \ref{subsec:LocallyFiniteInExtension} we show that with the same process we can get subdivisions of the polyhedron that are locally finite in $W(\Sigma)$ for $\Sigma$ a finite fan which is suitably compatible with the polyhedron. Readers only interested in local finiteness in $W$ may skip \S \ref{subsec:LocallyFiniteInExtension} and substitute references to Proposition \ref{prop:LocallyFiniteSubdivideExtendedPolyhedron} with Proposition \ref{prop:LocallyFiniteSubdividePolyhedron}. Then in \S \ref{subsec:SubdividingComplex} we show that we can piece together these subdivisions of individual polyhedra to get a subdivision of an appropriate polyhedral complex.

\begin{example}\label{example:One}
Let $N=\Z^2$ and consider the polyhedra 
\begin{align*}
P_1&=\{(x,y)\in N_{\R}\mid x\geq0,y\geq0\},\\
P_2&=\{(x,y)\in N_{\R}\mid x\leq0,y\geq0\},\\
P_3&=\{(x,y)\in N_{\R}\mid x\leq0,y\leq0\},\text{ and}\\
P_4&=\{(x,y)\in N_{\R}\mid x\geq0,y\leq0\}.
\end{align*}
Let $\smallcpx$ be the complex with maximal cells $P_i$ for $i\in\{2,3,4\}$, and let $\bigcpx$ be the complex with maximal cells $P_i$ for $i\in\{1,2,3,4\}$. The complexes $\smallcpx$ and $\bigcpx$ are illustrated below. 
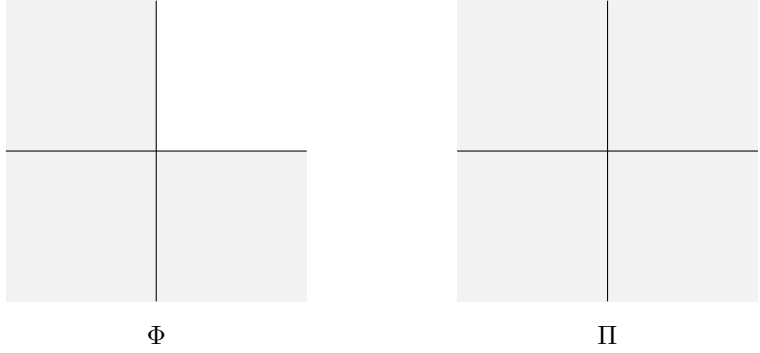
\begin{center}\begin{tikzpicture}[scale=1]
\fill[gray!10!white] (0,0) -- (0,2) -- (-2,2) -- (-2,0) -- cycle;
\fill[gray!10!white] (0,0) -- (0,-2) -- (-2,-2) -- (-2,0) -- cycle;
\fill[gray!10!white] (0,0) -- (0,-2) -- (2,-2) -- (2,0) -- cycle;
\draw(0,0) -- (0,2);
\draw(0,0) -- (0,-2);
\draw(0,0) -- (2,0);
\draw(0,0) -- (-2,0);
\node[below] at (0,-2.2) {$\smallcpx$};
\fill[gray!10!white] (6,0) -- (6,2) -- (8,2) -- (8,0) -- cycle;
\fill[gray!10!white] (6,0) -- (6,2) -- (4,2) -- (4,0) -- cycle;
\fill[gray!10!white] (6,0) -- (6,-2) -- (4,-2) -- (4,0) -- cycle;
\fill[gray!10!white] (6,0) -- (6,-2) -- (8,-2) -- (8,0) -- cycle;
\draw(6,0) -- (6,2);
\draw(6,0) -- (6,-2);
\draw(6,0) -- (8,0);
\draw(6,0) -- (4,0);
\node[below] at (6,-2.2) {$\bigcpx$};
\end{tikzpicture}\end{center}
For $i=1,\ldots,4$ let $\sigma_i$ be the recession cone of $P_i$; so as sets we have $\sigma_i=P_i$. Let $\smallfan$ be the fan whose maximal cones are $\sigma_2,\sigma_3,\sigma_4$. 
Intuitively, we want to specify $\widetilde{\bigcpx}$ by drawing the following picture.
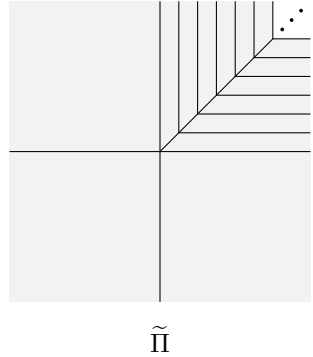
\begin{center}\begin{tikzpicture}[scale=1]
\fill[gray!10!white] (0,0) -- (0,2) -- (6/4,2) -- (6/4,6/4) -- (2,6/4) -- (2,0) -- cycle;
\fill[gray!10!white] (0,0) -- (0,2) -- (-2,2) -- (-2,0) -- cycle;
\fill[gray!10!white] (0,0) -- (0,-2) -- (-2,-2) -- (-2,0) -- cycle;
\fill[gray!10!white] (0,0) -- (0,-2) -- (2,-2) -- (2,0) -- cycle;
\draw (0,0) -- (0,2);
\draw (0,0) -- (0,-2);
\draw (0,0) -- (2,0);
\draw (0,0) -- (-2,0);
\foreach \x in {1,...,6}
	\draw(2,\x/4) -- (\x/4,\x/4) -- (\x/4,2);
\draw(0,0) -- (6/4,6/4);
\foreach \x in {1,...,3}
	\filldraw (6/4+2*\x/16,6/4+2*\x/16) circle (0.5pt);
\node[below] at (0,-2.2) {$\widetilde{\bigcpx}$};
\end{tikzpicture}\end{center}
We describe this subdivision as follows. Pick a rational ray $\rho$ in $N_{\R}$ which meets the relative interior of $\sigma_1$. Pick a nonzero point $\gamma$ on $\rho$ and let $T$ be the polyhedral complex obtained by subdividing $\rho$ at each of the non-negative integer multiples of $\gamma$. Let $\calB$ be the boundary complex of $P_1$, i.e., the set of all proper faces of $P_1$. Then the picture we have drawn above represents the polyhedral complex $\widetilde{\bigcpx}:=\smallcpx\cup\{Q+R\mid Q\in\calB,R\in T\}$. 
\end{example}

\subsection{Subdividing a polyhedron}\label{subsec:SubdividingPolyhedron}

The following two lemmas will allow us to construct subdivisions of a polyhedron using its boundary and a ray interior to the recession cone.

\begin{lemma}\label{lemma:RayBoundaryDecomp}
Let $P$ be an unbounded, pointed polyhedron in $W$. Fix a ray $\rho$ that meets the relative interior of $\rec P$. Let $B$ be the union of the faces of $P$ whose recession cones are proper faces of $\rec P$. Then any point in $P$ can be written uniquely as the sum of a point in $\rho$ and a point in $B$.
\end{lemma}
\begin{proof}
Let $w\in P$. Fix a nonzero point $v\in\rho$. Consider the function $\ph\fncolon\R\to W$ given by $t\mapsto w+tv$. Because $v$ is in $\rec P$, $\ph(t)\in P$ for $t\geq0$. On the other hand, because $P$ is a pointed polyhedron, $\ph^{-1}(P)$ is a polyhedron but not all of $\R$. Thus $\ph^{-1}(P)=[t_0,\infty)$ for some $t_0\leq0$.We have that $\{\ph(t_0)\}$ is a face of the polyhedron $\ph(\ph^{-1}(P))=\ph(\R)\cap P$. Since a face of an intersection of two polyhedra must be an intersection of faces of the polyhedra, there is a face $Q$ of $P$ such that $\{\ph(t_0)\}=\ph(\R)\cap Q$. The recession cone of $Q$ must be a face of $\rec P$ but does not contain $\rho$ because $Q$ contains $\ph(t_0)$ but no other point of $\ph(t_0)+\rho=\ph(\ph^{-1}(P))$. Since $\rho$ meets the relative interior of $\rec P$ this exactly says that the recession cone of $Q$ is a proper face of $P$. So $\ph(t_0)\in B$. Thus $w=\ph(t_0)+(-t_0v)$ is the desired representation.

To show uniqueness, suppose $z_1+v_1=z_2+v_2$ with $z_1,z_2\in B$ and $v_1,v_2\in\rho$. We may assume without loss of generality that $v_1-v_2\in\rho$. If $v_1=v_2$ then $z_1=z_2$ and we are done, so assume $v_1\neq v_2$. Thus $v_1-v_2\in\rho\sdrop\{0\}$, so $z_2=z_1+v_1-v_2$ is in the relative interior of the polyhedron $z_1+\rho$. Now if $Q$ is a face of $P$ containing $z_2$ then because $z_1+\rho\subset P$, $Q\cap(z_1+\rho)$ must be a face of $z_1+\rho$ containing $z_2$, and so must be all of $z_1+\rho$. Thus any such $Q$ contains $z_1+\rho$ which implies that the recession cone of $Q$ must contain $\rho$ and therefore be all of $\rec P$. This contradicts the assumption that $z_2\in B$.
\end{proof}

\begin{lemma}\label{lemma:ProjectionsOfSumAffine}
Let $F\subset W$ be a polyhedron and $\rho\subset W$ be a cone such that the addition map $\rho \times F \rightarrow F+P$ is a bijection. Let $\ph\fncolon F+\rho\to F$ and $\psi\fncolon F+\rho\to\rho$ be the maps such that $w=\ph(w)+\psi(w)$ for all $w\in F+\rho$. Then $\ph$ and $\psi$ are affine maps. In particular, $\ph$ and $\psi$ are continuous.
\end{lemma}\begin{proof}
Without loss of generality we can assume that $0$ is in the relative interior of $F$. In this case we see that the sum $\vspan F$ and $\vspan \rho$ is a direct sum. We thus have that the maps $\ph$ and $\psi$ are restrictions of the projection maps $(\vspan F)\oplus(\vspan \rho)\to\vspan F$ and $(\vspan F)\oplus(\vspan \rho)\to\vspan \rho$, respectively.
\end{proof}

We can now give a subdivision of a polyhedron by subdividing a ray intersecting the relative interior of the recession cone.

\begin{prop}\label{prop:SubdivideOnePolyhedron}
Let $P$ be an unbounded, pointed polyhedron in $W$. Fix a ray $\rho$ that meets the relative interior of $\rec P$ and let $B$ be the union of those faces of $P$ whose recession cones are proper faces of $\rec P$. Let $T$ be any polyhedral subdivision of $\rho$ and let $\Pi$ be a polyhedral complex with support $B$. Then $\Pi':=\{Q+R\mid Q\in\Pi, R\in T\}$ is a polyhedral subdivision of $P$.
\end{prop}\begin{proof}
Say $Q\in\Pi$ and $R\in T$. Lemma \ref{lemma:ProjectionsOfSumAffine} implies that the addition map $Q\times R\to Q+R$ is an affine isomorphism of polyhedra. Since faces of $Q\times R$ are products of a face of $Q$ and a face of $R$, the addition map $Q\times R\to Q+R$ being an affine isomorphism gives us that the faces of $Q+R$ are exactly those sets of the form $E+F$ with $E\leq Q$ and $F\leq R$. In particular, every face of $Q+R$ is also in $\Pi'$.

Say $Q_1,Q_2\in\Pi$ and $R_1,R_2\in T$ are such that $Q_1+R_1$ and $Q_2+R_2$ are not disjoint. 
The uniqueness part of Lemma \ref{lemma:RayBoundaryDecomp} shows that 
$$(Q_1+R_1)\cap(Q_2+R_2)=(Q_1\cap Q_2)+(R_1\cap R_2),$$ 
so $Q_1\cap Q_2\neq\emptyset$ and $R_1\cap R_2\neq\emptyset$. 
Since $Q_1\cap Q_2$ is a face of both $Q_1$ and $Q_2$, and $R_1\cap R_2$ is a face of both $R_1$ and $R_2$, the previous paragraph gives us that $(Q_1\cap Q_2)+(R_1\cap R_2)$ is a face of both $Q_1+R_1$ and $Q_2+R_2$. Thus $\Pi'$ is a polyhedral complex.
\end{proof}

We can also use Lemma \ref{lemma:ProjectionsOfSumAffine} to show that the construction used in Proposition \ref{prop:SubdivideOnePolyhedron} can be used to build subdivisions of $P$ which are locally finite in $W$.

\begin{prop}\label{prop:LocallyFiniteSubdividePolyhedron}
With the same setup as in Proposition \ref{prop:SubdivideOnePolyhedron}, if $T$ and $\Pi$ are both locally finite in $W$, then $\Pi'$ is locally finite in $W$ as well.
\end{prop}\begin{proof}
It follows from Lemma \ref{lemma:RayBoundaryDecomp} the addition map $+:\rho \times B\rightarrow P$ is a homeomorphism. Since $T$ is locally finite in $\rho$ and $\Pi$ is locally finite in $B$, the product family $T\times\Pi=\{R\times Q\mid R\in T, Q\in\Pi\}$ is locally finite in $\rho\times B$. Thus the family $\Pi'=\{R+Q\mid R\in T, Q\in\Pi\}$ of images of sets in $T\times\Pi$ under the homeomorphism $+\fncolon\rho\times B\to P$ is locally finite in $P$. Because $P$ is closed in $W$, this gives us that $\Pi'$ is locally finite in $W$.
\end{proof}


\subsection{Local finiteness in $W(\Sigma)$}\label{subsec:LocallyFiniteInExtension}

In proving Proposition \ref{prop:LocallyFiniteSubdividePolyhedron}, we used the fact that $P$ is closed in $W$. In order to use a similar proof to conclude that a polyhedral complex is locally finite in the tropical toric variety $W(\Sigma)$, we will need to consider the closure of $P$ in the tropical toric variety. 
We start by giving an explicit description of the closure of $P$ in $W(\Sigma)$ when $P$ and $\Sigma$ are suitably compatible. See Figure \ref{Fig:ClosureOfPolyhedron} for an explicit example. The rest of this subsection will be devoted to generalizing results from the previous subsection.

\begin{lemma}\label{lemma:ClosureOfPolyhedron}
Let $P\subset W$ be a pointed polyhedron with recession cone $\sigma:=\rec P$ and let $\Sigma$ be a finite fan in $W$ such that $\Sigma\cup\Faces(\sigma)$ is also a fan. Then the closure of $P$ in $W(\Sigma)$ is 
$$\nbar{P}=\bigcup_{\substack{\tau\leq\sigma\\\tau\in\Sigma}}\pi_{\tau}(P),$$
where $\pi_\tau\fncolon W\to W/\vspan\tau$ is the canonical projection map.
\end{lemma}\begin{proof}

In the case when $\Sigma$ contains $\sigma$, \cite[Proposition 3.19 and Remark 3.21]{RabinoffTropicalAnalyticGeometry} shows that $\nbar{P}=\bigcup_{\tau\leq\sigma}\pi_{\tau}(P)$. 
For an arbitrary $\Sigma$, the topology on $W(\Sigma)$ is the same as the subspace topology induced by the inclusion $W(\Sigma)\subset W(\Sigma\cup\Faces(\sigma))$, so the closure $\nbar{P}$ of $P$ in $W(\Sigma)$ is the intersection of $W(\Sigma)$ with the closure of $P$ in $W(\Sigma\cup\Faces(\sigma))$. 
By the first case we know that the closure of $P$ in $W(\Sigma\cup\Faces(\sigma))$ is $\bigcup_{\tau\leq\sigma}\pi_{\tau}(P)$, and the desired description of $\nbar{P}$ follows immediately.
\end{proof}

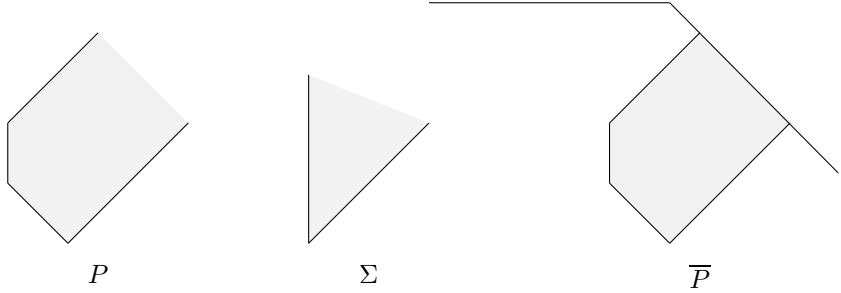
\begin{figure}
\begin{center}\begin{tikzpicture}[scale=0.8]
\fill[gray!10!white] (3,2) -- (1,0) -- (0,1) -- (0,2) -- (1.5,3.5) -- cycle;
\draw(3,2) -- (1,0);
\draw(1,0) -- (0,1);
\draw(0,1) -- (0,2);
\draw(0,2) -- (1.5,3.5);
\node[below] at (1.5,-0.2) {$P$};

\fill[gray!10!white] (5,0)-- (5,2.8) -- (7,2) -- cycle;
\draw(5,0)-- (5,2.8);
\draw(5,0) -- (7,2);
\node[below] at (6,-0.2) {$\Sigma$};

\fill[gray!10!white] (13,2) -- (11,0) -- (10,1) -- (10,2) -- (11.5,3.5) -- cycle;
\draw(13,2) -- (11,0);
\draw(11,0) -- (10,1);
\draw(10,1) -- (10,2);
\draw(10,2) -- (11.5,3.5);
\node[below] at (11.5,-0.2) {$\nbar{P}$};

\draw(7,4) -- (11,4);
\draw(11,4) -- (13.8,1.17);

\end{tikzpicture}\end{center}
\caption{The closure of $P$ in $W(\Sigma)$, the lines in third figure represent the boundary of $W$ in $W(\Sigma)$.}\label{Fig:ClosureOfPolyhedron}
\end{figure}

We now generalize Lemma \ref{lemma:RayBoundaryDecomp} and Lemma \ref{lemma:ProjectionsOfSumAffine} for the closure of a polyhedron in  $W(\Sigma)$.

\begin{lemma}\label{lemma:RayBoundaryClosureDecomp}
Let $P\subset W$ be a unbounded, pointed polyhedron with recession cone $\sigma:=\rec P$ and let $\Sigma$ be a finite fan in $W$ such that $\Sigma\cup\Faces(\sigma)$ is also a fan. Let $\rho$ be a ray that meets the relative interior of $\sigma$ and let $B$ be the union of those faces of $P$ whose recession cones are strictly smaller than $\sigma$. 
Then the addition map $+\fncolon W\times W(\Sigma)\to W(\Sigma)$ sends $\rho\times\nbar{B}$ to $\nbar{P}$. 
Moreover, if $\Sigma$ does not contain $\sigma$ then the restricted addition map $\rho\times\nbar{B}\toup{+}\nbar{P}$ is a bijection.
\end{lemma}\begin{proof}
First we describe $\nbar{B}$. Let $\calB$ be the set of faces of $P$ whose recession cones are strictly smaller than $\sigma$. By Lemma \ref{lemma:ClosureOfPolyhedron} we have
\begin{equation}\label{eqn:DescribeBBar}
\nbar{B}=\bigcup_{F\in\calB}\nbar{F}
=\bigcup_{F\in\calB}\bigcup_{\substack{\tau\leq\rec F\\\tau\in\Sigma}}\pi_\tau(F)
\end{equation}
where $\nbar{F}$ is the closure of $F$ in $W(\Sigma)$.

Now we show that the addition map $+\fncolon W\times W(\Sigma)\to W(\Sigma)$ sends $\rho\times\nbar{B}$ to $\nbar{P}$. Notice that if $w\in \nbar{B}$ then by (\ref{eqn:DescribeBBar}) $w\in \pi_{\tau}(F)$ for some $F\in \calB$ and $\tau\leq \rec F$. If $v\in \rho$ then  we can compute $v+w$ by taking $z\in F \subseteq B$ such that $\pi_{\tau}(z)=w$ and noting that $v+w=\pi_{\tau}(v+z)$. We have $\pi_{\tau}(v+z)\in \pi_{\tau}(P)$ by Lemma \ref{lemma:RayBoundaryDecomp} and by Lemma \ref{lemma:ClosureOfPolyhedron} we have $\pi_{\tau}(P)\subseteq \nbar{P}$.

Finally, we show that $\rho\times\nbar{B}\toup{+}\nbar{P}$ is a bijection when $\Sigma$ does not contain $\sigma$. For any $w\in \nbar{B}$ and $v\in \rho$, $v+w\in W/\vspan \tau$ if and only if $w\in W/\vspan \tau$. Thus we can proceed by showing that for every $\tau\neq \sigma$ the map $\rho \times (\nbar{B} \cap W/\vspan \tau ) \rightarrow \nbar{P}\cap W/ \vspan \tau$ is a bijection. By Lemma \ref{lemma:ClosureOfPolyhedron} we have that $\overline{P}\cap W/\vspan \tau=\pi_{\tau}(P)$. If $\tau\not\leq\sigma$ then Lemma \ref{lemma:ClosureOfPolyhedron} gives us that $\nbar{B}\cap W/\vspan\tau\subset\nbar{P}\cap W/\vspan\tau=\emptyset$, so suppose $\tau<\sigma$. In this case we will show that $\pi_{\tau}(P)$ is a pointed and unbounded polyhedron in $W/\vspan \tau$, that $\nbar{B}\cap W/\vspan \tau$ is the union of faces of $\pi_{\tau}(P)$ whose recession cones are proper faces of $\rec \pi_{\tau}(P)$, and that $\pi_{\tau}(\rho)$ is a ray meeting the relative interior of $\pi_{\tau}(P)$. Thus we can apply Lemma \ref{lemma:RayBoundaryDecomp} to the map $\rho \times (\nbar{B} \cap W/\vspan \tau ) \rightarrow \nbar{P}\cap W/ \vspan \tau$.

Because $\tau < \sigma$ it follows from Lemma \ref{lemma:MapAndRecessionCommute} it follows that $\pi_{\tau}(P)$ is pointed and unbounded. 

We also have that
$$\nbar{B}\cap W/\vspan\tau=\dcup_{\substack{F\in\calB\\ \tau\leq\rec F}}\pi_\tau(F).$$
Lemma \ref{lemma:FaceModRecessionGivesFace} gives that $\pi_\tau(F)$ is a face of $\pi_{\tau}(P)$ when $F\in \calB $ and $\tau \leq \rec F$. Lemma \ref{lemma:MapAndRecessionCommute} gives that, for each such face $F$, $\pi_{\tau}(F)$ is a face whose recession cone is a proper face of $\rec \pi_{\tau}(P)$. To see that these are all such faces of $\pi_{\tau}(P)$ consider a face $F'$ of $\pi_{\tau}(P)$ satisfying $\rec F'<\pi_{\tau}(\sigma)$. Then $F:=P\cap\pi_{\tau}^{-1}(F')$ is a face of $P$ such that $\pi_{\tau}(F)=F'$. We have $\tau\subset\rec F$ because $\tau$ is contained in both $\sigma$ and $\ker\pi_{\tau}$ and so $F+\tau=F$, and further $\tau\leq\rec F$ because both $\tau$ and $\rec F$ are faces of $\sigma$.

Finally, we have 
$$\vspan(\tau)\cap\rho=\vspan(\tau)\cap\sigma\cap\rho=\tau\cap\rho=\{0\},$$
from which we get that $\pi_\tau$ maps $\rho$ bijectively onto its image. In particular, $\pi_\tau(\rho)$ is a ray. Note that $\pi_\tau(\rho)\subset\pi_\tau(\sigma)=\rec(\pi_\tau(P))$ by Lemma \ref{lemma:MapAndRecessionCommute}. Further, $\pi_\tau(\rho)$ meets the relative interior of $\pi_\tau(\sigma)$, for if not $\pi_\tau(\rho)$ would be contained in a proper face $\sigma'$ of $\pi_\tau(\sigma)$, but then $\rho$ would be contained in $\pi_\tau^{-1}(\sigma')\cap\sigma$ which is a proper face of $\sigma$.

Thus we can apply Lemma \ref{lemma:RayBoundaryDecomp} to see that the map $\rho \times (\nbar{B} \cap W/\vspan \tau ) \rightarrow \nbar{P}\cap W/ \vspan \tau$ is a bijection.
\end{proof}

\begin{lemma}\label{lemma:RayBoundaryHomeo}
Using the same notation as above, assume that $\Sigma$ does not contain $\sigma$. Then the addition map $+\fncolon\rho\times\nbar{B}\to\nbar{P}$ is a homeomorphism.
\end{lemma}\begin{proof}
First, observe that we can reduce to the case where $\Sigma$ is the set $\properFaces{\sigma}:=\Faces(\sigma)\sdrop\{\sigma\}$ of proper faces of $\sigma$. Indeed, by Lemma \ref{lemma:ClosureOfPolyhedron} we know that $\nbar{B}\subseteq W(\Sigma)$ is the same as the closure $\cl_{W(\Sigma)\cap W(\sigma)}(B)$ of $B$ in $W(\Sigma)\cap W(\sigma)$ which is contained in the closure $\cl_{W(\properFaces{\sigma})}(B)$ of $B$ in $W(\properFaces{\sigma})$. If we know that the addition map gives a homeomorphism $\rho\times\cl_{W(\properFaces{\sigma})}(B)\to\cl_{W(\properFaces{\sigma})}(P)$, then the restriction of the addition map to $\rho\times\nbar{B}$ is a homoemomorphism onto its image, which is equal to $\nbar{P}$ by Lemma \ref{lemma:RayBoundaryClosureDecomp}.

Now consider the case in which $\Sigma = \properFaces{\sigma}$. Let $\calB$ be the set of faces of $P$ whose recession cones are in $\properFaces{\sigma}$, so $\nbar{B}=\bigcup_{F\in\calB}\nbar{F}$ where $\nbar{F}$ is the closure of $F$ in $W(\properFaces{\sigma})$. By Lemma \ref{lemma:ClosureOfPolyhedron} $\nbar{F}$ is the same as the closure of $F$ in $W(\rec F)$, and so by \cite[3.18.1]{RabinoffTropicalAnalyticGeometry} $\nbar{F}$ is compact. Thus $\nbar{B}$, being a finite union of compact sets, is compact.

Fix a nonzero point $v$ in $\rho$, and for each positive integer $n$ let 
$$R_n:=\{tv\mid t\in[0,n]\}.$$ 
Each $R_n$ is compact, and the interior of $R_n$ as a subset of $\rho$ is $R_n^{\circ}:=\{tv\mid t\in[0,n)\}$ so the collection of all $R_n^{\circ}$ covers $\rho$. Thus $\calC:=\{R_n\times\nbar{B}\mid n>0\}$ is a collection of compact subsets of $\rho\times\nbar{B}$ and the collection of all $R_n^{\circ}\times \nbar{B}$ covers $\rho\times\nbar{B}$. Since the addition map $+\fncolon\rho\times\nbar{B}\to\nbar{P}$ is a bijection, we have that the collection of all $R_n^{\circ}+\nbar{B}$ covers $\nbar{P}$. Furthermore each $U_n:=R_n^\circ+\nbar{B}$ is an open subset of $\nbar{P}$. 

From Lemma \ref{lemma:TropicalToricVarsHausdorff} we know that $\nbar{P}\subset W(\properFaces{\sigma})$ is Hausdorff. Furthermore the collection $\{U_n \mid n > 0 \}$ is an open cover of $\nbar{P}$. Thus  it follows from\cite[Ch.\ 1, \S9.4, Corollary 2]{BourbakiGenTopI} the addition map $\rho\times\nbar{B}\toup{+}\nbar{P}$ is a homeomorphism.

\end{proof}

We can prove the main result of this subsection.

\begin{prop}\label{prop:LocallyFiniteSubdivideExtendedPolyhedron}
Let $P\subset W$ be a unbounded, pointed polyhedron with recession cone $\sigma:=\rec P$ and let $\Sigma$ be a finite fan in $W$ not containing $\sigma$ but such that $\Sigma\cup\Faces(\sigma)$ is also a fan. Let $\rho$ be a ray that meets the relative interior of $\sigma$ and let $B$ be the union of those faces of $P$ whose recession cones are proper faces of $\sigma$. Suppose that $T$ is a polyhedral subdivision of $\rho$ which is locally finite in $\rho$ and $\Pi$ is a polyhedral complex with support $B$ such that $\Pi$ is locally finite in $W(\Sigma)$. Then $\Pi':=\{Q+R\mid Q\in\Pi, R\in T\}$ is a polyhedral subdivision of $P$ which is locally finite in $W(\Sigma)$.
\end{prop}\begin{proof}
By Proposition \ref{prop:SubdivideOnePolyhedron}, $\Pi'$ is a polyhedral subdivision of $P$.

Lemma \ref{lemma:RayBoundaryHomeo} tells us that the addition map $+\fncolon\rho\times \nbar{B}\to\nbar{P}$ is a homeomorphism where $\nbar{B}$ and $\nbar{P}$ are the closures of $B$ and $P$, respectively, in $W(\Sigma)$. Since $T$ is locally finite in $\rho$ and $\Pi$ is locally finite in $\nbar{B}\subset W(\Sigma)$, the product family $T\times\Pi=\{R\times Q\mid R\in T, Q\in\Pi\}$ is locally finite in $\rho\times\nbar{B}$. Thus the family $\Pi'=\{R+Q\mid R\in T,Q\in\Pi\}$ of images of sets in $T\times\Pi$ under the homeomorphism $+\fncolon\rho\times \nbar{B}\to\nbar{P}$ is locally finite in $\nbar{P}$. Because $\nbar{P}$ is closed in $W(\Sigma)$, this gives us that $\Pi'$ is locally finite in $W(\Sigma)$.
\end{proof}


\subsection{Subdividing a polyhedral complex}\label{subsec:SubdividingComplex}

We will now take the results from the previous subsection and use them to construct subdivisions of an entire complex.

\begin{lemma}\label{lemma:PutTogetherSubdivision}
Let $\Pi$ be a polyhedral complex in $W$ and let $\widetilde{\Pi}$ be a collection of polyhedra such that each $Q\in\widetilde{\Pi}$ is contained in some $P\in\Pi$ and for each $P\in\Pi$, $\{Q\in\widetilde{\Pi}\mid Q\subset P\}$ is a subdivision of $P$. Then $\widetilde{\Pi}$ is a polyhedral complex subdividing $\Pi$. 
\end{lemma}\begin{proof}
If $Q\in\widetilde{\Pi}$, pick $P\in\Pi$ such that $Q\subset P$. Then $Q$ is in the polyhedral complex $\{S\in\widetilde{\Pi}\mid S\subset P\}$, so every face of $Q$ is also in $\{S\in\widetilde{\Pi}\mid S\subset P\}\subset\widetilde{\Pi}$.

If $Q_1,Q_2\in\widetilde{\Pi}$ are not disjoint, pick $P_1,P_2\in\Pi$ such that $Q_1\subset P_1$ and $Q_2\subset P_2$. We have that $F:=P_1\cap P_2$ is a face of both $P_1$ and $P_2$, and in particular $F\in\Pi$. For $i=1,2$ we have that $Q_i\subset P_i$ and that $F$ is a face of $P_i$ meeting $Q_i$, which implies that $Q_i\cap F$ is a face of $Q_i$. In particular, this gives us that $Q_1\cap F,Q_2\cap F\in\widetilde{\Pi}$, and so $Q_1\cap F$ and $Q_2\cap F$ are both in the polyhedral complex $\{S\in\widetilde{\Pi}\mid S\subset F\}$. Thus $Q_1\cap Q_2=(Q_1\cap F)\cap(Q_2\cap F)$ is a face of both $Q_1\cap F$ and $Q_2\cap F$, and so is also a face of both $Q_1$ and $Q_2$.
\end{proof}

We now describe our subdivision technique. If the following conditions hold we will say that we are in situation $(\star)$.
\begin{itemize}
\item Let $\smallfan$ be a finite fan in $W$.
\item Let $\smallcpx$ be a polyhedral complex in $W$ such that $\rec P\in\smallfan$ for all $P\in\smallcpx$.
\item Let $\bigcpx$ be a polyhedral complex in $W$ which contains $\smallcpx$ as a subcomplex, is locally finite in $W(\smallfan)$, and is such that $\bigfan:=\smallfan\cup\{\rec P\mid P\in\bigcpx\}$ is a fan. 
\item For each $\sigma\in\bigfan\sdrop\smallfan$ pick a ray $\rho_\sigma$ that meets the relative interior of $\sigma$.
\item Let $T_\sigma$ be a polytopal subdivision of $\rho_\sigma$ which is locally finite in $\rho_\sigma$. 
\end{itemize}

\begin{construction}\label{construction:Subdivisions}
Say we are in situation $(\star)$. For each $\sigma\in\bigfan$ let 
$$\bigcpx_\sigma:=\{P\in\bigcpx\mid\rec P=\sigma\},$$ 
and for $P\in\bigcpx_\sigma$ let $B_P$ be the union of those faces of $P$ whose recession cones are proper faces of $\sigma$. We recursively define a sequence $\left\{\widetilde{\bigcpx}(i)\right\}_{i\in\Z_{\geq0}}$ of polyhedral complexes in $W$ as follows. Let $\widetilde{\bigcpx}(0):=\bigcpx_{\{0\}}$. For $i\geq1$ assume we are given $\widetilde{\bigcpx}(i-1)$. 
For every $\sigma\in\bigfan\sdrop \smallfan$ such that $\dim \sigma=i$ and every $P\in\bigcpx_\sigma$, let $\calS(P):=\left\{Q+R\;\middle|\; Q\in\widetilde{\bigcpx}(i-1),Q\subset B_P, R\in T_\sigma\right\}$. 
Finally we define
$$\widetilde{\bigcpx}(i):=\widetilde{\bigcpx}(i-1)\cup\bigg( \bigcup_{\substack{\sigma\in\smallfan\\ \dim\sigma=i}}\bigcpx_{\sigma}\bigg)\cup \bigg(\bigcup_{\substack{\sigma\in\bigfan\sdrop\smallfan\\ \dim\sigma=i}}\bigcup_{P\in\bigcpx_\sigma}\calS(P)\bigg).$$
We also let $\widetilde{\bigcpx}:=\widetilde{\bigcpx}(\dim W)$.
\end{construction}

\begin{proposition}\label{prop:MainSubdivisions}
Say we are in situation $(\star)$. 
Then $\widetilde{\bigcpx}$, as defined in Construction \ref{construction:Subdivisions}, is a polyhedral subdivision of $\bigcpx$ which is locally finite in $W(\smallfan)$ such that 
$\smallcpx\subset\widetilde{\bigcpx}$ and 
$\{\rec P\mid P\in\widetilde{\bigcpx}\}=\smallfan\cap\{\rec P\mid P\in\bigcpx\}$.
\end{proposition}

Before we prove Proposition \ref{prop:MainSubdivisions}, we consider an example to help the reader build intuition for Construction \ref{construction:Subdivisions}.

\begin{example}\label{example:Three}
Let $W=\R^2$ and consider the polyhedra
\begin{align*}
P_1&=\{(x,y)\in \R^2\mid x\geq 1,y\geq0\},\\
P_2&=\{(x,y)\in \R^2\mid 1\geq x\geq 0,y\geq0\},\\
P_3&=\{(x,y)\in \R^2\mid x\leq 0,y\geq0\},\\
P_4&=\{(x,y)\in \R^2\mid x\leq 0,y\leq0\},\\
P_5&=\{(x,y)\in \R^2\mid 1\geq x\geq 0,y\leq0\},\text{ and}\\
P_6&=\{(x,y)\in \R^2\mid x\geq 1,y\leq0\}.\\
\end{align*}
Let $\smallcpx$ be the polyhedral complex with with maximal cells $P_j$ for $j\in\{4,5,6\}$ and let $\bigcpx$ be the complex with maximal cells $P_j$ for $j\in\{1,2,3,4,5,6\}$. The complexes $\smallcpx$ and $\bigcpx$ are drawn in the following pictures.
\begin{center}\begin{tikzpicture}[scale=\scalefactor]
\fill[gray!10!white] (0,0) -- (0,-2) -- (-2,-2) -- (-2,0) -- cycle;
\fill[gray!10!white] (.5,0) -- (.5,-2) -- (2.5,-2) -- (2.5,0) -- cycle;
\fill[gray!10!white] (0,0) -- (0,-2) -- (.5,-2) -- (.5,0) -- cycle;
\draw(0,0) -- (0,-2);
\draw(.5,0) -- (.5,-2);
\draw(0,0) -- (.5,0);
\draw(.5,0) -- (2.5,0);
\draw(0,0) -- (-2,0);
\node[below] at (.25,-2.2) {$\smallcpx$};
\fill[gray!10!white] (6.5,0) -- (6.5,2) -- (8.5,2) -- (8.5,0) -- cycle;
\fill[gray!10!white] (6,0) -- (6,2) -- (6.5,2) -- (6.5,0) -- cycle;
\fill[gray!10!white] (6,0) -- (6,2) -- (4,2) -- (4,0) -- cycle;
\fill[gray!10!white] (6,0) -- (6,-2) -- (4,-2) -- (4,0) -- cycle;
\fill[gray!10!white] (6.5,0) -- (6.5,-2) -- (8.5,-2) -- (8.5,0) -- cycle;
\fill[gray!10!white] (6,0) -- (6,-2) -- (6.5,-2) -- (6.5,0) -- cycle;
\draw(6,0) -- (6,2);
\draw(6.5,0) -- (6.5,2);
\draw(6,0) -- (6,-2);
\draw(6.5,0) -- (6.5,-2);
\draw(6,0) -- (6.5,0);
\draw(6.5,0) -- (8.5,0);
\draw(6,0) -- (4,0);
\node[below] at (6.25,-2.2) {$\bigcpx$};
\end{tikzpicture}\end{center}
Let $\smallfan:=\{\rec P\mid P\in\smallcpx\}$ and $\bigfan:=\{\rec P\mid P\in\bigcpx\}$. 
Note that $\bigfan\sdrop\smallfan$ consists of the three cones $\sigma_1:=\{(x,y)\mid x\geq 0,y\geq 0\}$, $\sigma_2:=\{(0,y)\mid y\geq 0\}$ and $\sigma_3:=\{(x,y)\mid x\leq 0, y\geq 0\}$. 
We have that $\sigma_2$ is a ray so $\rho_{\sigma_2}=\sigma_2$. For particular choices of $\rho_{\sigma_1},\rho_{\sigma_3},T_{\sigma_1},T_{\sigma_2},$ and $T_{\sigma_3}$, the following pictures show $\{P\in\bigcpx\mid \dim(\rec P)\leq i\}$ and its subdivision $\widetilde{\bigcpx}(i)$ for $i\in\{0,1,2\}$.
%
%
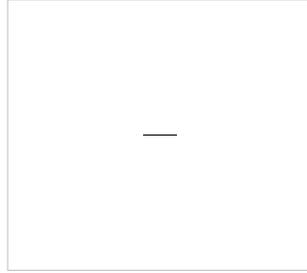
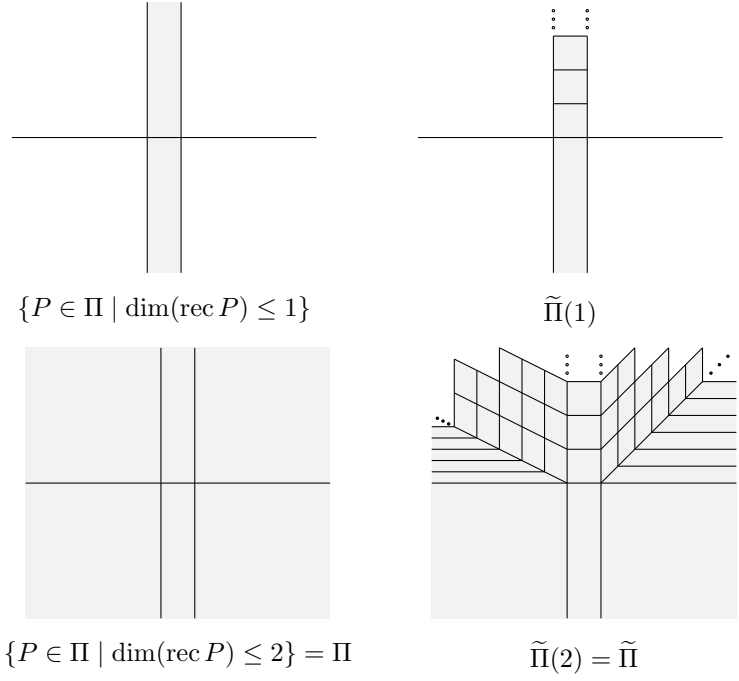
\begin{center}\begin{tikzpicture}[scale=\scalefactor]
\draw(0,0) -- (0.5,0);
\draw[lightgray, very thin] (-2,-2) -- (2.5,-2) -- (2.5,2) -- (-2,2) -- cycle;
\node[below] at (0.25,-2.2) {$\{P\in\bigcpx\mid\dim(\rec P)\leq 0\}=\widetilde{\bigcpx}(0)$};
\end{tikzpicture}\end{center}
%
%
\begin{center}\begin{tikzpicture}[scale=\scalefactor]
\fill[gray!10!white] (0,0) -- (0,2) -- (.5,2) -- (.5,0) -- cycle;
\fill[gray!10!white] (0,0) -- (0,-2) -- (0.5,-2) -- (0.5,0) -- cycle;
\draw(0,0) -- (0,2);
\draw(0.5,0) -- (0.5,2);
\draw(0,0) -- (0,-2);
\draw(0.5,0) -- (0.5,-2);
\draw(0,0) -- (0.5,0);
\draw(0.5,0) -- (2.5,0);
\draw(0,0) -- (-2,0);
\node[below] at (0.25,-2.2) {$\{P\in\bigcpx\mid\dim(\rec P)\leq 1\}$};
\fill[gray!10!white] (6+0,0) -- (6+0,3/2) -- (6+.5,3/2) -- (6+.5,0) -- cycle;
\fill[gray!10!white] (6+0,0) -- (6+0,-2) -- (6+.5,-2) -- (6+.5,0) -- cycle;
\draw (6+0,0) -- (6+0,3/2);
\draw (6+0.5,0) -- (6+0.5,3/2);
\foreach \x in {1,2,3}
	\draw (6+0,3/2+\x/8) circle (0.5pt);
\foreach \x in {1,2,3}
	\draw (6+0.5,3/2+\x/8) circle (0.5pt);
\draw (6+0,0) -- (6+0,-2);
\draw (6+0.5,0) -- (6+0.5,-2);
\draw (6+0,0) -- (6+.5,0);
\draw (6+0.5,0) -- (6+2.5,0);
\draw (6+0,0) -- (6-2,0);
\foreach \x in {1,2,3}
	\draw (6+0,\x/2) -- (6+.5,\x/2);
\node[below] at (6+0.25,-2.2) {$\widetilde{\bigcpx}(1)$};
\end{tikzpicture}\end{center}
%
%
\begin{center}\begin{tikzpicture}[scale=\scalefactor]
\fill[gray!10!white] (.5,0) -- (0.5,2) -- (2.5,2) -- (2.5,0) -- cycle;
\fill[gray!10!white] (0,0) -- (0,2) -- (.5,2) -- (.5,0) -- cycle;
\fill[gray!10!white] (0,0) -- (0,2) -- (-2,2) -- (-2,0) -- cycle;
\fill[gray!10!white] (0,0) -- (0,-2) -- (-2,-2) -- (-2,0) -- cycle;
\fill[gray!10!white] (.5,0) -- (0.5,-2) -- (2.5,-2) -- (2.5,0) -- cycle;
\fill[gray!10!white] (0,0) -- (0,-2) -- (0.5,-2) -- (0.5,0) -- cycle;
\draw(0,0) -- (0,2);
\draw(0.5,0) -- (0.5,2);
\draw(0,0) -- (0,-2);
\draw(0.5,0) -- (0.5,-2);
\draw(0,0) -- (0.5,0);
\draw(0.5,0) -- (2.5,0);
\draw(0,0) -- (-2,0);
\node[below] at (0.25,-2.2) {$\{P\in\bigcpx\mid\dim(\rec P)\leq 2\}=\bigcpx$};
\fill[gray!10!white] (6+0.5,0) -- (6+0.5,3/2) -- (6+1/2+.5,2) -- (6+1/2+.5,3/2) -- (6+1+.5,2) -- (6+1+.5,3/2) -- (6+6/4+.5,2) -- (6+6/4+.5,6/4) -- (6+2+.5,6/4) -- (6+2+.5,0) -- cycle;
\fill[gray!10!white] (6+0,0) -- (6+0,3/2) -- (6-1,2) -- (6-1,3/2) -- (6-5/3,11/6) -- (6-5/3,5/6) -- (6-2,5/6) -- (6-2,0) -- cycle;
\fill[gray!10!white] (6+0,0) -- (6+0,-2) -- (6-2,-2) -- (6-2,0) -- cycle;
\fill[gray!10!white] (6+0.5,0) -- (6+0.5,-2) -- (6+2.5,-2) -- (6+2.5,0) -- cycle;
\fill[gray!10!white] (6+0,0) -- (6+0,3/2) -- (6+.5,3/2) -- (6+.5,0) -- cycle;
\fill[gray!10!white] (6+0,0) -- (6+0,-2) -- (6+.5,-2) -- (6+.5,0) -- cycle;
\draw (6+0,0) -- (6+0,3/2);
\draw (6+0.5,0) -- (6+0.5,3/2);
\foreach \x in {1,2,3}
	\draw (6+0,3/2+\x/8) circle (0.5pt);
\foreach \x in {1,2,3}
	\draw (6+0.5,3/2+\x/8) circle (0.5pt);
\draw (6+0,0) -- (6+0,-2);
\draw (6+0.5,0) -- (6+0.5,-2);
\draw (6+0,0) -- (6+.5,0);
\draw (6+0.5,0) -- (6+2.5,0);
\draw (6+0,0) -- (6-2,0);
\foreach \x in {2,4,6}
	\draw(6+2.5,\x/4) -- (6+\x/4+.5,\x/4) -- (6+\x/4+.5,2);
\foreach \x in {1,3,5}
	\draw(6+2+.5,\x/4) -- (6+\x/4+.5,\x/4) -- (6+\x/4+.5,7/4);
\draw(6+0.5,0) -- (6+6/4+.5,6/4);
\draw(6+0.5,1/2) -- (6+6/4+.5,2);
\draw(6+0.5,1) -- (6+1+.5,2);
\draw(6+0.5,3/2) -- (6+1/2+.5,2);
\foreach \x in {1,2,3}
	\filldraw (6+6/4+2*\x/16+.5,6/4+2*\x/16) circle (0.5pt);
\foreach \x in {1,2,3}
	\draw (6+0,\x/2) -- (6+.5,\x/2);
%
%
%
%
\foreach \x in {1,4}
	\draw(6-\x/3,10/6) -- (6-\x/3,\x/6) -- (6-2,\x/6);
	\foreach \x in {2,5}
	\draw(6-\x/3,11/6) -- (6-\x/3,\x/6) -- (6-2,\x/6);
\draw(6-1,2) -- (6-1,1/2) -- (6-2,1/2);
\foreach \x in {0,1,2}
	\draw (6+0,\x/2) -- (6-5/3,5/6+\x/2);
\draw (6+0,3/2) -- (6-1,2);
\foreach \x in {1,2,3}
	\filldraw (6-5/3-\x/12,5/6+\x/24) circle (0.5pt);
\node[below] at (6+0.25,-2.2) {$\widetilde{\bigcpx}(2)=\widetilde{\bigcpx}$};
\end{tikzpicture}\end{center}
\end{example}

\begin{proof}[Proof of Proposition \ref{prop:MainSubdivisions}]
We prove by induction on $i\in\Z_{\geq0}$ that:
\begin{enumerate}
\item\label{itemThm2pf:PolyhedralComplex} $\widetilde{\bigcpx}(i)$ is a polyhedral complex,
\item\label{itemThm2pf:LocallyFinite} $\widetilde{\bigcpx}(i)$ is locally finite in $W(\Sigma)$,
\item\label{itemThm2Pf:Subdivision} $\widetilde{\bigcpx}(i)$ subdivides the polyhedral complex $\Pi (i):=\{P\in\bigcpx\mid \dim(\rec P)\leq i\}$,
\item\label{itemThm2Pf:SigmaSubset} $\{P\in \smallcpx\mid \dim(\rec P)\leq i\}\subset\{P\in\bigcpx\mid \rec P\in\smallfan,\dim(\rec P)\leq i\}\subset\widetilde{\bigcpx}(i)$, and
\item\label{itemThm2Pf:RecFan} $\{\rec P\mid P\in\widetilde{\bigcpx}(i)\}$ is the set $\smallfan\cap\{\rec P\mid P\in\bigcpx,\dim(\rec P)\leq i\}$.
\end{enumerate}
The theorem will then follow by plugging in $i=\dim W$. The base case $i=0$ is clear. For the inductive step, let $i\geq 1$.

Part (\ref{itemThm2pf:PolyhedralComplex}) follows from Lemma \ref{lemma:PutTogetherSubdivision}. We will verify that for any $P\in \bigcpx(i)$ the collection $\{Q\in \widetilde{\bigcpx}(i) \mid Q\subseteq P\}$ is a polyhedral complex, by the inductive hypothesis we know this is the case for $\dim (\rec P) < i$. Let $P\in \bigcpx(i)$, $\dim (\rec  P) =i$. Consider the case when $\rec P\in \Sigma$. Then for $F$ a face of $P$, $\rec F \in \Sigma$; therefore $\{Q\in \widetilde{\bigcpx}(i) \mid Q\subseteq P\}$ consists of the faces of $P$ which is a polyhedral complex. Now consider the case when $\rec P \in\bigfan\sdrop\smallfan$. Then we have that  $\{Q\in \widetilde{\bigcpx}(i) \mid Q\subseteq P\}= \calS (P)$, which is a polyhedral subdivision by Proposition \ref{prop:SubdivideOnePolyhedron}.

To see that (\ref{itemThm2pf:LocallyFinite}) note that
$\widetilde{\bigcpx}(i-1)$ is locally finite in $W(\Sigma)$, and  because $\Pi$ is locally finite in $W(\Sigma)$, then the collection of $\Pi_{\sigma}$, $\sigma \in \Sigma '$ and $\dim \sigma =i $ is locally finite in $W(\Sigma)$. By Proposition \ref{prop:LocallyFiniteSubdivideExtendedPolyhedron} we have that $\calS (P)$ is locally finite in $W(\Sigma)$. It follows then that $ \widetilde{\bigcpx}(i)$ is locally finite in $W(\Sigma)$.

Part (\ref{itemThm2Pf:Subdivision}) follows from the fact that for $P\in \Pi(i)$ either $P\in \widetilde{\bigcpx}(i)$ or $\widetilde{\bigcpx}(i)$ contains $\calS (P)$. We know $\calS (P)$ is a subdivision of $P$ by Proposition \ref{prop:SubdivideOnePolyhedron}.

Parts (\ref{itemThm2Pf:SigmaSubset}) and (\ref{itemThm2Pf:RecFan}) follow immediately from the inductive hypothesis and the construction of $\widetilde{\bigcpx}(i)$.

\end{proof}


\section{Locally finite completions of polyhedral complexes}\label{sec:LocallyFiniteCompletions}

We now use Proposition \ref{prop:MainSubdivisions} to deduce the existence of completions of polyhedral complexes with various properties. 
In \S \ref{section:GeneralSubdivisionCompletionTheorems} we prove general theorems on the existence of locally finite subdivisions and completions. We then apply them to our main theorems.

\subsection{General subdivision and completion theorems}\label{section:GeneralSubdivisionCompletionTheorems}

To state our subdivision results in full generality we require a few technical definitions. 
Let $\calP$ be a class of polyhedra in $W$ and let $\calC=\calC(\calP)$ be the class of recession cones of polyhedra $P\in\calP$. 
A \emph{$\calP$-complex} is a polyhedral complex consisting of polyhedra in $\calP$. The definitions of a $\calC$-fan, a $\calP$-completion, and a $\calP$-polytopal subdivision are analogous. 
If $\Sigma$ is a $\calC$-fan then we define $(\calP,\Sigma)$ to be the class of those $P\in\calP$ with $\rec P\in\Sigma$.

\begin{defi}
We say that $\calP$ satisfies condition $(\dagger)$ if 
\begin{enumerate}[(I)]
\item\label{itemConditionDaggerPart:MinkowskiSum} $\calP$ is closed under taking Minkowski sums,
\item\label{itemConditionDaggerPart:RayExists} for every nonzero $\sigma\in\calC$ there is a ray $\rho\in\calC$ which meets the relative interior of $\sigma$,
\item\label{itemConditionDaggerPart:SubdividingRay} for every ray $\rho\in\calC$ there is a $\calP$-polytopal subdivision of $\rho$ which is locally finite in $\rho$.
\end{enumerate}
\end{defi}

Using this definition, we can state our general subdivision theorem.

\begin{thm}\label{thm:GeneralSubdivisions}
Let $\calP$ be a class of polyhedra in $W$ which satisfies condition $(\dagger)$. 
Let $\smallfan$ be a finite $\calC$-fan.
Suppose that $\smallcpx\subset\bigcpx$ is an extension of $\calP$-complexes which are both locally finite in $W(\smallfan)$.
Assume $\rec P\in\smallfan$ for $P\in\smallcpx$ and $\bigfan:=\smallfan\cup\{\rec P\mid P\in\bigcpx\}$ is a fan. Then there is a $\calP$-subdivision $\widetilde{\bigcpx}$ of $\bigcpx$ which is locally finite in $W(\smallfan)$ such that $\smallcpx\subset\widetilde{\bigcpx}$ and $\{\rec P\mid P\in\widetilde{\bigcpx}\}=\smallfan\cap\{\rec P\mid P\in\bigcpx\}$.
\end{thm}
\begin{proof}
For each $\sigma\in\bigfan\sdrop\smallfan$ pick a ray $\rho_{\sigma}\in\calC$ which meets the relative interior of $\sigma$ and a $\calP$-polytopal subdivision $T_{\sigma}$ of $\rho_\sigma$ which is locally finite in $\rho_\sigma$. 
This puts us in situation $(\star)$. Let $\widetilde{\bigcpx}(i)$ for $i\in\Z_{\geq0}$, $\bigcpx_{\sigma}$, and $\calS (P)$ be defined as in Construction \ref{construction:Subdivisions}. 
Induction on $i$ shows that every $S\in\widetilde{\bigcpx}(i)$, $S\in \calP$. Indeed, the base case is true by assumption, as $\widetilde{\bigcpx}(0)= \bigcpx_{\{0\}}\subseteq \smallcpx $. For $i>0$ we have by definition:
$$
\widetilde{\bigcpx}(i):=\widetilde{\bigcpx}(i-1)\cup\bigg(\bigcup_{\substack{\sigma\in\smallfan\\ \dim\sigma=i}}\bigcpx_{\sigma}\bigg)\cup\bigg(\bigcup_{\substack{\sigma\in\bigfan\sdrop\smallfan\\ \dim\sigma=i}}\bigcup_{P\in\bigcpx_\sigma}\calS(P)\bigg).
$$
By the inductive step the elements of $\widetilde{\bigcpx}(i-1)$ are contained in $\calP$. By definition $\bigcpx_{\sigma}\subseteq \smallcpx$ and so all its elements are in $\calP$ by assumption. Any element of $\calS (P)$ is a Minkowski sum of an element of $\widetilde{\bigcpx}(i-1)$, and an element of $T_{\sigma}$ for some $\sigma$. Thus by assumption any element of $\calS (P)$ is in $\calP$. Thus each $S\in\widetilde{\bigcpx}(i)$ is in $\calP$. 
By Proposition \ref{prop:MainSubdivisions}, $\widetilde{\bigcpx}:=\widetilde{\bigcpx}(\dim W)$ is the desired $\calP$-subdivision of $\bigcpx$.
\end{proof}

Condition $(\dagger)$ is easy to verify in many cases.

\begin{prop}\label{prop:CasesSatisfyingEasyConditionDagger}
Let $N$ be a lattice. The following classes of polyhedra in $N_{\R}$ satisfy condition $(\dagger)$.
\begin{enumerate}
\item\label{itemEasyConditionCase:SubfieldRational} The class of all $\subfieldOfR$-definable polyhedra for a fixed subfield $\subfieldOfR$ of $\R$. 
\item\label{itemEasyConditionCase:GammaRational} The class of all $\Gamma$-rational polyhedra for a fixed nontrivial additive subgroup $\Gamma$ of $\R$.
\item\label{itemEasyConditionCase:GammaRationalWVertices} The class of all $\Gamma$-rational polyhedra with vertices in $N_{\Gamma}$ for a fixed nontrivial additive subgroup $\Gamma$ of $\R$.

\end{enumerate}
\end{prop}
\begin{proof}
(\ref{itemEasyConditionCase:SubfieldRational}):
For (\ref{itemConditionDaggerPart:MinkowskiSum}), the fact that the Minkowski sum of two $\subfieldOfR$-definable polyhedra is $\subfieldOfR$-definable follows from the description of $\subfieldOfR$-definable polyhedra as Minkowski sums of $\subfieldOfR$-definable polytopes and $\subfieldOfR$-definable pointed cones.
For (\ref{itemConditionDaggerPart:RayExists}), consider a nonzero $\subfieldOfR$-definable cone $\sigma$. Because $\F$ is dense in $\R$ it follows that the intersection of $N_{\subfieldOfR}=N\otimes_{\Z}\subfieldOfR$ with the relative interior of $\sigma$ is nonempty.
For (\ref{itemConditionDaggerPart:SubdividingRay}) given an $\subfieldOfR$-definable ray $\rho$, write $\rho=\R_{\geq0}v$ for some nonzero $v\in N_{\subfieldOfR}$. Then
$$\{nv\mid n\in\Z_{\geq0}\}\cup\{\conv(nv,(n+1)v)\mid n\in\Z_{\geq0}\}$$
is a polytopal subdivision of $\rho$, which is locally finite in $\rho$, and whose elements are $\subfieldOfR$-definable. 

(\ref{itemEasyConditionCase:GammaRational}):
For (\ref{itemConditionDaggerPart:MinkowskiSum}), we know $\calP$ is closed under taking Minkowski sums. 
For (\ref{itemConditionDaggerPart:RayExists}) note that $\calC$ is the class of all rational cones in $N_{\R}$, so (\ref{itemConditionDaggerPart:RayExists}) follows from (\ref{itemEasyConditionCase:SubfieldRational}) with $\subfieldOfR=\Q$. 
For (\ref{itemConditionDaggerPart:SubdividingRay}), if $\rho$ is a ray in $\calC$ then $\rho=\R_{\geq0}v$ for some $v\in N$, so if we fix $\gamma\in\Gamma$ then 
$$\{n\gamma v\mid n\in\Z_{\geq0}\}\cup\{\conv(n\gamma v,(n+1)\gamma v)\mid n\in\Z_{\geq0}\}$$
is a $\Gamma$-rational polytopal subdivision of $\rho$ which is locally finite in $\rho$.

(\ref{itemEasyConditionCase:GammaRationalWVertices}):
Part (\ref{itemConditionDaggerPart:MinkowskiSum}) follows from (\ref{itemEasyConditionCase:GammaRational}) and the fact that $N_{\Gamma}$ is a group. Parts (\ref{itemConditionDaggerPart:RayExists}) and (\ref{itemConditionDaggerPart:SubdividingRay}) follow from (\ref{itemEasyConditionCase:GammaRational}) as a sufficiently large integer mutltiple of an element of $N_{\Q\Gamma}$ is in $N_{\Gamma}$.
\end{proof}

In order to state our general result on locally finite completions, we need one more definition.

\begin{defi}
We say that $\calP$ \emph{admits finite recession-restricted completions} if, whenever $\smallfan$ is a finite $\calC$-fan and $\smallcpx$ is a finite $(\calP,\smallfan)$-complex, there is a finite $\calP$-completion, $\bigcpx$, of $\smallcpx$ such that $\{\rec P\mid P\in \bigcpx\}\cup\smallfan$ is a fan.
\end{defi}

\begin{thm}\label{thm:GeneralLocallyFiniteCompletions}
Let $\calP$ be a class of polyhedra in $W$ which has satisfies condition $(\dagger)$ and admits finite recession-restricted completions. Then, for any finite $\calC$-fan $\smallfan$, any finite $(\calP,\smallfan)$-complex $\cpxToComplete$ has a $\calP$-completion $\completeCpx$ which is locally finite in $W(\smallfan)$ and satisfies $\{\rec P\mid P\in\completeCpx\}=\smallfan$.
\end{thm}
\begin{proof}
Because $\calP$ admits finite recession-restricted completions, $\cpxToComplete$ admits a finite completion $\bigcpx$ such that $\{\rec P\mid P\in\bigcpx\}\cup\smallfan$ is a fan. 
Since $\bigcpx$ is a finite complete complex, $\{\rec P\mid P\in\bigcpx\}$ is a complete complex of cones \cite[Theorem 3.4]{RecessionFan}. 
Because $\{\rec P\mid P\in\bigcpx\}\cup\smallfan$ is also a complex, this gives us that $\{\rec P\mid P\in\bigcpx\}$ contains $\smallfan$. 
Because $\calP$ satisfies condition $(\dagger)$, Theorem \ref{thm:GeneralSubdivisions} gives us that there is a $\calP$-subdivision $\subdividedCpx$ of $\bigcpx$ which is locally finite in $W(\smallfan)$ such that $\smallcpx\subset\subdividedCpx$ and $\{\rec P\mid P\in\subdividedCpx\}=\smallfan\cap\{\rec P\mid P\in\bigcpx\}=\smallfan$. 
Thus $\completeCpx:=\subdividedCpx$ is as desired.
\end{proof}

In order to apply Theorem \ref{thm:GeneralLocallyFiniteCompletions} to particular cases we will want to show that various classes of polyhedra satisfy condition $(\dagger)$ and admit finite recession-restricted completions. 
Proposition \ref{prop:CasesSatisfyingEasyConditionDagger} tells us that certain classes of polyhedra satisfy condition $(\dagger)$, and to see that some of these classes also admit finite recession-restricted completions we will use completion results from the literature. 
The completion results are stated in terms of fans, and so we will use the following lemma to translate these results into statements about complexes. First, given $P$ a polyhedron in $W$ let
$$c(P):=\nbar{\{(tw,t)\mid t\in\R_{\geq0}, w\in P\}}\subseteq N_{\R}\times\R_{\geq0}.$$

\begin{lemma}\label{lemma:ComplexToFan}
Let $\Pi$ be a polyhedral complex in $W$ and let $\Sigma$ be a fan in $W$ such that for all $P\in\Pi$, $\rec P\in\Sigma$. Then $\{c(P)\mid P\in\Pi\}\cup\{\sigma\times\{0\}\mid\sigma\in\Sigma\}$ is a fan in $W\times\R_{\geq0}$.
\end{lemma}\begin{proof}
This is essentially proven in the proof of Theorem 3.4(2) of \cite{RecessionFan}.
\end{proof}

\begin{remark}
In Example 3.1 of \cite{RecessionFan} it is shown that one cannot take an arbitrary polyhedral complex and construct a fan from cones over the polyhdera and their recession cones. Note that the example given their does \emph{not} satisfy the hypotheses of the above lemma.
\end{remark}

\begin{prop}\label{prop:CasesAdmitingFiniteCompletions}
Let $N$ be a lattice. The following classes of polyhedra in $N_{\R}$ satisfy condition $(\dagger)$ and admit finite recession-restricted completions.
\begin{enumerate}
\item\label{itemFiniteCompletionsCase:SubfieldRational} The class of all $\subfieldOfR$-definable polyhedra for a fixed subfield $\subfieldOfR$ of $\R$. 
\item\label{itemFiniteCompletionsCase:GammaRational}
The class of all $\Gamma$-rational polyhedra for a fixed nontrivial additive subgroup $\Gamma$ of $\R$.
\end{enumerate}
\end{prop}\begin{proof}
By Proposition \ref{prop:CasesSatisfyingEasyConditionDagger} the classes $\calP$ above satisfy condition $(\dagger)$. 
Given a finite $\calC$-fan $\smallfan$ in $N_{\R}$ and a finite $(\calP,\smallfan)$-complex $\smallcpx$, let 
$$\halfspaceSmallFan:=\{c(P)\mid P\in\smallcpx\}\cup\{\sigma\times\{0\}\mid\sigma\in\smallfan\},$$ 
which is a fan in $N_{\R}\times\R_{\geq0}$ by Lemma \ref{lemma:ComplexToFan}. We now consider the cases separately.

(\ref{itemFiniteCompletionsCase:SubfieldRational}) The fan $\halfspaceSmallFan$ is finite and $\subfieldOfR$-definable, so it admits a finite $\subfieldOfR$-definable completion $\halfspaceBigFan$ in $N_{\R}\times\R$; see \cite[Theorem 5.4]{EwaldIshida} or \cite[Theorem 5.3]{Rohrer}. 
Then the complex $\bigcpx:=\htone{\halfspaceBigFan}{N_{\R}}$ is a finite $\subfieldOfR$-definable completion of $\smallcpx=\htone{\halfspaceSmallFan}{N_{\R}}$. 
For $P\in\bigcpx$ we have $\rec P\in\htzero{\halfspacebigfan}{N_{\R}}$, and \cite[Lemma 3.5]{RecessionFan} gives us that any face of a cone $\sigma\in\{\rec P\mid P\in\bigcpx\}$ is also in $\{\rec P\mid P\in\bigcpx\}$, so $\{\rec P\mid P\in\bigcpx\}$ is a subfan of $\htzero{\halfspacebigfan}{N_{\R}}$. Since $\smallfan=\htzero{\halfspaceSmallFan}{N_{\R}}$ is also a subfan of $\htzero{\halfspacebigfan}{N_{\R}}$, $\{\rec P\mid P\in\bigcpx\}\cup\smallfan$ is a subfan of $\htzero{\halfspacebigfan}{N_{\R}}$. In particular, $\{\rec P\mid P\in\bigcpx\}\cup\smallfan$ is a fan.

(\ref{itemFiniteCompletionsCase:GammaRational}) The fan $\halfspaceSmallFan$ in $N_{\R}\times\R_{\geq0}$ is finite. Furthermore it is $\Gamma$-admissible in the sense of \cite{GublerGuideTrop}. So \cite[Theorem 1.2]{GammaAdmissibleCompletions} tells us that $\halfspaceSmallFan$ admits a finite $\Gamma$-admissible completion $\halfspaceBigFan$ in $N_{\R}\times\R_{\geq0}$. Then $\bigcpx:=\htone{\halfspaceBigFan}{N_{\R}}$ is a finite $\Gamma$-rational completion of $\smallcpx=\htone{\halfspaceSmallFan}{N_{\R}}$. As in part (\ref{itemFiniteCompletionsCase:SubfieldRational}), we get that $\{\rec P\mid P\in\bigcpx\}\cup\smallfan$ is a fan.
\end{proof}

We can now prove the first of our main theorems.

\begin{proof}[Proof of Theorem \ref{thm:LocallyFiniteCompletions}]
The class of $\Gamma$-rational polyhedra satisfies condition $(\dagger)$ by Proposition \ref{prop:CasesSatisfyingEasyConditionDagger} part (\ref{itemEasyConditionCase:GammaRational}). The class of $\Gamma$-rational polyhedra also admits finite recession-restricted completions by Proposition \ref{prop:CasesAdmitingFiniteCompletions}. So the result follows by Theorem \ref{thm:GeneralLocallyFiniteCompletions}. If $\Phi$ admits a completion with vertices in $N_{\Gamma}$  it follows from Proposition \ref{prop:CasesSatisfyingEasyConditionDagger} part (\ref{itemEasyConditionCase:GammaRationalWVertices}) and Theorem \ref{thm:GeneralSubdivisions} that $\nbar{\Phi}$ can be taken to have vertices in $N_{\Gamma}$. 
\end{proof}

\begin{thm}\label{thm:LocallyFiniteKDefinableCompletions}
Fix a lattice $N$ and a subfield $\subfieldOfR$ of $\R$. Let $\smallfan$ be 
a finite $\subfieldOfR$-definable fan 
in $N_{\R}$, and let $\cpxToComplete$ be a finite $\subfieldOfR$-definable polyhedral complex in $N_{\R}$ such that $\rec P\in\smallfan$ for all $P\in\cpxToComplete$. Then $\cpxToComplete$ has an $\subfieldOfR$-definable completion $\completeCpx$ which is locally finite in $N_{\R}(\smallfan)$ and satisfies $\{\rec P\mid P\in\completeCpx\}=\smallfan$. 
\end{thm}\begin{proof}
This follows from Theorem \ref{thm:GeneralLocallyFiniteCompletions} and Proposition \ref{prop:CasesAdmitingFiniteCompletions} (\ref{itemFiniteCompletionsCase:SubfieldRational}).
\end{proof}

\begin{proof}[Proof of Theorem \ref{thm:LocallyFiniteRealPolytopalCompletions}]
This is the particular case of Theorem \ref{thm:GeneralLocallyFiniteCompletions} where $\smallfan$ is the zero fan and $\calP$ is the class of polyhedra in $W$, which satisfies $(\dagger)$ as Minkowski sums of polytopes are polytopes and the other two conditions are vacuous as polytopes are bounded.
\end{proof}

\subsection{Zonotopal completions}\label{section:ZonotopalCompletions}

In this section we use Theorem \ref{thm:GeneralSubdivisions} to prove Theorem \ref{thm:ZonotopalCompletions}. 
Before starting the proof of Theorem \ref{thm:ZonotopalCompletions} we define the notion of a star-shaped ball in $W$.

\begin{defi}
A subset $D$ of $W$ is called a \emph{star-shaped ball} if there is a point $w\in W$ such that every ray $\rho$ starting at $w$ intersects $\bdry D$ in a unique point $v\neq w$ and $D\cap\rho=\conv(w,v)$. We say that $w$ is a \emph{center} for $D$ and that $D$ is a \emph{star-shaped ball around $w$}.
\end{defi}
\noindent Here $\bdry D$ denotes the topological boundary of $D$ as a subset of $W$.

We follow \cite{EwaldSchulz} and \cite[Ch.\ III.5]{EwaldBook} in using this definition, which is particularly suited to several operations on polyhedral complexes, even though it is somewhat stronger than another common notion of \emph{star-shaped region}. We refer to \cite{EwaldSchulz} and \cite[Ch.\ III.5]{EwaldBook} for relevant facts about star-shaped balls.

The case of Theorem \ref{thm:ZonotopalCompletions} where $|\cpxToComplete|$ is a polytope will be reduced to the case where $|\cpxToComplete|$ is a star-shaped ball. This case will be proved by constructing a certain finite completion of $\cpxToComplete$ and then applying Theorem \ref{thm:GeneralSubdivisions}. The construction of the finite completion of $\cpxToComplete$ works in a broader generality, and so we start with this construction.

Let $\cpxToComplete$ be a finite polytopal complex in $W$ such that $D:=|\cpxToComplete|$ is a star-shaped ball around $0$. For any $F\in\cpxToComplete$ with $F\subset \bdry D$, let $U(F):=\R_{\geq1}F=\{tv\mid t\geq 1, v\in F\}$. We immediately get that $U(F)\cap D=F$. Since $F$ is a polyhedron contained in $\bdry D$, $\aff(F)$ does not contain 0. So there is some $u\in W^*$ and a positive real number $a$ such that $\angbra{u,v}=a$ for all $v\in F$. Note that $U(F)=(\R_{\geq0}F)\cap\{v\in W\mid \angbra{u,v}\geq a\}$, and is therefore a polyhedron. Also, because every face of an intersection of two polyhedra can be written as an intersection of faces of the original polyhedra, and because the faces of $\R_{\geq0}F$ are exactly the zero cone and $\R_{\geq0}G$ for $G$ a face of $F$, we have that the faces of $U(F)$ are exactly the faces of $F$ and the polyhedra $U(G)$ where $G$ is a face of $F$.

\begin{lemma}\label{lemma:ExtendingStarshapedComplex}
For any finite polytopal complex $\cpxToComplete$ in $W$ such that $D:=|\cpxToComplete|$ is a star-shaped ball around $0$, $\cpxToComplete\cup\{U(F)\mid F\in\cpxToComplete, F\subset\bdry D\}$ is a complete polyhedral complex, where $U(F):=\R_{\geq1}F$.
\end{lemma}\begin{proof}
Given what we have established in the preceding paragraph, the proof is a straightforward analogue of the standard proof that, under the same hypotheses, $\{\R_{\geq0}F\mid F\in\cpxToComplete, F\subset\bdry D\}\cup\{0\}$ is a complete fan. We leave the details to the reader.
\end{proof}

 Let $\Pzon$ be the class of polyhedra $P\subset W$ such that every bounded face of $P$ is a zonotope. In order to use Theorem \ref{thm:GeneralSubdivisions} to obtain an appropriate subdivision of the completion constructed above, we need to show that $\Pzon$ satisfies $(\dagger)$.

\begin{lemma}\label{lemma:PzonSatsDagger}
The class $\Pzon$ satisfies condition $(\dagger)$.
\end{lemma}\begin{proof}
For $P,Q\in\Pzon$, any bounded face of $P+Q$ is the Minkowski sum of a bounded face $F$ of $P$ and a bounded face $G$ of $Q$. Then $F$ and $G$ are zonotopes so $F+G$ is also a zonotope. 
The class $\calC_{zon}=\{\rec P\mid P\in\Pzon\}$ consists of all cones in $W$, so (\ref{itemConditionDaggerPart:RayExists}) and (\ref{itemConditionDaggerPart:SubdividingRay}) follow from the fact that any ray is contained in $\calC$ and a subdivision of a ray is consists of points, line segments, and possibly a ray. 
\end{proof}

We now have all of the ingredients we need to prove Theorem \ref{thm:ZonotopalCompletions}.

\begin{proof}[Proof of Theorem \ref{thm:ZonotopalCompletions}]
First suppose that $|\cpxToComplete|$ is a star-shaped ball $D$. By shifting we may assume without loss of generality that $0$ is a center for $D$. Lemma \ref{lemma:ExtendingStarshapedComplex} tells us that $\bigCpx:=\cpxToComplete\cup\{U(F)\mid F\in\cpxToComplete, F\subset\bdry D\}$ is a complete polyhedral complex in $W$. If $F\in\cpxToComplete$ is contained in $\bdry D$ then every bounded face of $U(F)$ is also a face of $F$, and is therefore a zonotope. Thus $\bigCpx$ is a $\Pzon$-complex. Since $\cpxToComplete\subset\bigCpx$ consists of polytopes we have that the zero cone is in $\{\rec P\mid P\in\bigCpx\}$, so by \cite[Corollary 3.10]{RecessionFan}, $\{\rec P\mid P\in\bigCpx\}$ is a fan. 
By Lemma \ref{lemma:PzonSatsDagger} we can apply Theorem \ref{thm:GeneralSubdivisions} with $\calP=\Pzon$ and $\smallfan$ being the zero fan to get that there is a $\Pzon$-subdivision $\completeCpx$ of $\bigCpx$ which is locally finite in $W(\smallfan)=W$ such that $\cpxToComplete\subset\completeCpx$ and $\{\rec P\mid P\in\completeCpx\}$ is the zero fan. 
Because $\bigCpx$ is complete, $\completeCpx$ is a completion of $\cpxToComplete$. Finally, $\{\rec P\mid P\in\completeCpx\}$ being the zero fan gives us that $\completeCpx$ is a polytopal $\Pzon$-complex, i.e., a zonotopal complex.

Now suppose that $|\cpxToComplete|$ is a polytope $P$. If $|\cpxToComplete|$ is a full-dimensional polytope in $W$, then $|\cpxToComplete|$ is a star-shaped ball with each point in the interior of $|\cpxToComplete|$ being a center, so the desired completion exists by the previous case. If $P$ is not full-dimensional then we can reduce to the full-dimensonal case as follows. By shifting we may assume without loss of generality that 0 is in the affine span of $P$, so $\aff(P)$ is a linear subspace of $W$. Let $V\subset W$ be a linear complement to $\aff(P)$, so $W=\aff(P)\oplus V$. Let $Q$ be a full-dimensional cube in $V$ with 0 as a vertex. Then the direct sum of $\cpxToComplete$ and the complex of all faces of $Q$ is a finite zonotopal subdivision of the full-dimensional polytope $P+Q$ in $W$ which contains $\cpxToComplete$ as a subcomplex. 
\end{proof}

\bibliographystyle{alpha}
\bibliography{LocallyFinite}

\end{document}